\documentclass[11pt]{article}

\usepackage{amssymb}
\usepackage{amsmath,amsfonts,mathtools}
\usepackage{amsthm}
\usepackage{latexsym}
\usepackage{mathrsfs}
\usepackage{xcolor}
\usepackage{float}
\usepackage{enumitem}
\usepackage{algorithm,algorithmic}
 \allowdisplaybreaks[3]
 
 \usepackage{bbm}
\usepackage{microtype}
\usepackage[a4paper]{geometry}
\usepackage{pifont}
\usepackage{tikz}
\usepackage{latexsym}

\usepackage[a4paper]{geometry}
\usepackage{xcolor}
\usepackage{graphicx}
\usepackage{tikz}
\usetikzlibrary{arrows.meta,fit,backgrounds,positioning}

\newcommand{\beq}{\begin{equation}}
\newcommand{\eeq}{\end{equation}}
\newcommand{\beqa}{\begin{eqnarray}}
\newcommand{\eeqa}{\end{eqnarray}}
\newcommand{\beqas}{\begin{eqnarray*}}
\newcommand{\eeqas}{\end{eqnarray*}}
\newcommand{\ba}{\begin{array}}
\newcommand{\ea}{\end{array}}
\newcommand{\bi}{\begin{itemize}}
\newcommand{\ei}{\end{itemize}}

\newcommand{\mcA}{{\mathcal A}}
\newcommand{\mcE}{{\mathcal E}}
\newcommand{\mcN}{{\mathcal N}}

\newcommand{\argmin}{\arg\min}
\newcommand{\argmax}{\arg\max}

\newtheorem{lemma}{Lemma}
\newtheorem{thm}{Theorem}

\newtheorem{defi}{Definition}
\newtheorem{prop}{Proposition}

\newtheorem{rem}{Remark}

\newcounter{spb}
\def\tr0{{\tilde r_0}}

\def\rr{{\mathbb{R}}}

\def\bfx{{\bf x}}

\def\bfp{{\bf p}}
\def\bfv{{\bf v}}
\def\bfc{{\bf c}}
\def\bfe{{\bf e}}
\def\bfb{{\bf b}}
\def\bfr{{\bf r}}
\def\bfzero{{\bf 0}}

\def\mcS{{\mathcal S}}
\def\Cyc{\mathrm{Cyc}}
\def\Path{\mathrm{Path}}
\def\mcG{\mathcal{G}}
\def\bfP{{\bf P}}
\def\mcC{\mathcal{C}}

\title{The Simple Strategy-Iteration Method is Strongly Polynomial for the Turn-Based Deterministic Forward Game}
\author{
Sanyou Mei
\thanks{Department of Industrial Engineering and Decision Analytics, Hong Kong University of Science and Technology, Hong Kong, China (email: {\tt symei@ust.hk, yyye@ust.hk}).}
\and
Chunlin Sun
\thanks{Institute for Computational and Mathematical Engineering, Stanford University, California, USA. (email: {\tt chunlin@stanford.edu}).}
\and
Yinyu Ye \footnotemark[1]
\thanks{Professor Emeritus of Department of Management Science and Engineering, Stanford University, California, USA.}
\thanks{Visiting Chair Professor of Department of Data and Business Intelligence, Shanghai Jiao Tong University, Shanghai, China; Visiting Scholar of Shanghai Institute for Mathematics and Interdisciplinary Sciences, Shanghai, China.}
}

\begin{document}

\maketitle

\begin{abstract}

We study Turn-Based Deterministic Forward Games (TBDFGs), the subclass of turn-based deterministic zero-sum games in which no directed cycle contains actions controlled by both players. This forward condition is strictly weaker than acyclicity: recurrent behavior may be arbitrarily rich within one player's states, while mixed-player feedback cycles are excluded. Our main contribution separates two algorithmic consequences of this structure. First, we analyze the simple strategy-iteration method of \cite{hansen2013strategy,jia2020towards}, a generic method for TBSGs whose execution neither tests for nor uses the TBDFG property. We prove that this structure-oblivious algorithm nevertheless has a strongly polynomial guarantee on every TBDFG. In particular, it terminates after at most $O(n^6m^4\log^4 n)$ simplex pivot steps. Thus, the forward property acts as a structural certificate for convergence even when the algorithm is not informed that the input has this property. Second, when the TBDFG structure is known in advance, a backward SCC propagation algorithm is proposed that solves a sequence of deterministic-MDP subproblems and improves the bound to $O(n^3m^2\log^2 n)$ simplex pivot steps. Together, these results show that forward structure both regularizes the convergence of a general strategy-iteration method and supports a sharper structure-aware algorithm.
\end{abstract}

\medskip


\section{Introduction}

A Turn-Based Deterministic (Zero-Sum) Game (TBDG) is a special case of a Turn-Based Stochastic (Zero-Sum) Game (TBSG) in which the transition probability distribution associated with every action is deterministic (see \cite{shapley1953stochastic}). A TBSG extends a Markov decision problem (MDP) by partitioning the states between two players: one player minimizes a common linear objective, while the other maximizes it. Although finite MDPs are polynomial-time solvable through linear programming and its deterministic version is strongly polynomial time solvable \cite{post2015simplex}, the existence of strongly polynomial algorithms for broad classes of TBSGs and TBDGs remains open.

As in an MDP, at each time step the Markov process is in a state $s=1,\dots,n$, and the decision maker may choose an action $a$ belonging to state $s$, denoted by $a\in\mcA_s$. Such a selection of one action per state is called a strategy $\pi$. After an action is chosen, the process moves to the next state and incurs an immediate cost $c_a$. The transition probabilities are determined by the chosen action $a$ through a probability distribution vector $\bfp_a\geq\bfzero\in\rr^n$. If $\bfp_a$ is deterministic, or equivalently a point-mass distribution, for every $a$, then the TBSG reduces to a TBDG. In this case, the process moves along a directed graph whose nodes are states and whose directed edges are chosen actions. In the strategy graph, each action is either on a path, called a path action, or on a cycle, called a cycle action.

In a TBSG, the total set of $n$ states and their actions are partitioned into two sets, $\mcS^1$ and $\mcS^2$. The player controlling $\mcS^1$ minimizes the common total discounted cost over the infinite horizon, while the player controlling $\mcS^2$ maximizes it. We refer to these players as Player 1 and Player 2, respectively, and let $|\mcS^1|=n_1$ and $|\mcS^2|=n_2$ with $n_1+n_2=n$. Furthermore, let $k_s$ be the number of actions available in state $s$, and suppose the actions are sorted as
\[
\mcA_s=\left\{a\in\mathbb{N}:\sum_{j=0}^{s-1}k_j< a\leq \sum_{j=0}^sk_j\right\}\quad\forall s=1,\dots,n.
\]
Here $\mathbb{N}$ is the set of natural numbers and $k_0=0$. We denote $\mcA^i=\cup_{s\in \mcS^i}\mcA_s$ for $i=1,2$. Then the total number of actions is $m=\sum_{s=1}^nk_i$, where Player 1 controls $m_1=\sum_{s=1}^{n_1}k_s$ actions $\mcA^1$ and Player 2 controls $m_2=m-m_1$ actions in $\mcA^2$.

We now give the formal definition of TBDGs.
\begin{defi}[{\bf TBDG}]
A \emph{2-Player Deterministic Zero-Sum Game} is \emph{defined by a tuple}  $\mcG=(\mcS,\mcA,\bfc,\bfP,\gamma)$. The elements of the tuple are defined as follows.
\begin{enumerate}
    \item $\mcS = \mcS^1 \cup \mcS^2 = [n]$ denotes the state space, where $\mcS^1$ and $\mcS^2$ are sets of states controlled by Players 1 and 2, respectively, with $\mcS^1 \cap\mcS^2 = \phi$.
    \item $\mcA = \cup_{s \in\mcS} \mcA_s = [m]$ denotes the action space, where $\mcA_s$ is the set of actions playable at state $s$. We assume $\mcA_s \cap \mcA_{s'} = \phi$ for all $s,s' \in\mcS$ and $s\neq s'$, and denote $\mcA^i = \cup_{s \in\mcS^i} \mcA_s$ as the set of actions available to Player $i$, for $i=1,2$.
    \item $\bfc$ is the cost vector, where $c_a$ denotes the immediate cost when action $a \in \mcA$ is taken.
    \item $P$ is the probability transition matrix. After action $a \in \mcA_s$ is taken at state $s \in \mcS$, the next state is chosen according to $\bfp_a$, where $\bfp_a$ is point-mass distribution for all $a \in \mcA$. 
    \item $\gamma\in(0,1)$ is a fixed discount factor. If an infinite sequence of actions $\{a_k\}_{k=0}^{\infty}$ is taken, the total discounted cost of this action sequence is $\sum_{k=0}^{\infty} \gamma^k c_{a_k}.$ 
\end{enumerate}
\end{defi}

We next define strategies and optimal strategies.

\begin{defi}[{\bf Strategy \& Strategy profile}]
A strategy for Player $i$ is a vector $\pi^i$ such that Player $i$ chooses action $\pi^i_s$ at state $s\in\mcS^i$, for $i=1,2$. A strategy profile $\pi=(\pi^1,\pi^2)$ is a pair of strategies, one for each player.
\end{defi}
Each strategy profile can be viewed as a subgraph of the underlying directed graph of the corresponding TBDG. Specifically, each state $s\in\mcS$ is a vertex in the subgraph, and the corresponding action $\pi_s$ is an out-going edge. The out-degree of each vertex is $1$, and each vertex and edge lies either on a cycle or on a path in this subgraph. For clarity, we specify the following definitions.

\begin{defi}[{\bf Directed Cycle}]
A directed cycle in a directed graph is a non-empty directed trail in which only the first and last vertices are equal.
\end{defi}

\begin{defi}
For any strategy profile $\pi=(\sigma,\tau)$, we denote by $\Path(\pi)$ the set of paths and by $\Cyc(\pi)$ the set of cycles in the corresponding subgraph. With a slight abuse of notation, we also regard the states and actions appearing on these paths and cycles as elements of $\Path(\pi)$ and $\Cyc(\pi)$, respectively. 
\end{defi}

By definition, a strategy profile $\pi=(\pi^1,\pi^2)$ of a TBSG specifies one action for Player 1 in each state of $\mcS^1$, denoted by $\pi^1=\{\pi_1,\dots,\pi_{n_1}\}$, and one action for Player 2 in each state of $\mcS^2$, denoted by $\pi^2=\{\pi_{n_1+1},\dots,\pi_n\}$. An optimal strategy profile is one in which no player can improve by switching actions at any state when optimizing the expected discounted sum over the infinite horizon with discount factor $\gamma$. Here, $\gamma>0$ is assumed to be strictly less than one but may be arbitrarily close to one. It is known that an optimal strategy profile exists for TBSGs.

Let $v_s$ denote the cost-to-go value of state $s$ under a strategy profile $\pi$ of a TBSG. Then $\pi$ is an optimal strategy profile if and only if $\bfv^*=(v^*_1,\dots,v^*_n)$ is a fixed point of the following operator \cite{hansen2013strategy,shapley1953stochastic}:
\begin{equation*}
\begin{aligned}
&v^*_s:=\min_{a\in\mcA_s}\left\{c_a+\gamma\bfp^T_a\bfv^*\right\}\quad\forall s\in \mcS^1,\\
&v^*_s:=\max_{a\in\mcA_s}\left\{c_a+\gamma\bfp^T_a\bfv^*\right\}\quad\forall s\in \mcS^2,
\end{aligned}
\end{equation*}
and the corresponding optimal strategy profile is
\begin{align*}
&\pi^*_s:=\argmin_{a\in\mcA_s}\left\{c_a+\gamma\bfp^T_a\bfv^*\right\}\quad\forall s\in \mcS^1,\\
&\pi^*_s:=\argmax_{a\in\mcA_s}\left\{c_a+\gamma\bfp^T_a\bfv^*\right\}\quad\forall s\in \mcS^2.
\end{align*}
The following lemma and definitions are known; see, for example, \cite{hansen2013strategy,shapley1953stochastic}.
\begin{lemma}[{\bf Value vectors}] \label{lemma1}
Let $\pi$ be any strategy profile, let $P^\pi$ be the $n$ by $n$ matrix where $P^\pi_{s,s'}$ is the probability of transitioning from $s'$ to $s$ using the actions in $\pi$, and let $\bfc^\pi$ be the vector of costs for each state under $\pi$. Then the state-indexed cost-to-go value vector $\bfv^\pi$ satisfies
\begin{equation*}
\label{def-v}
\bfv^\pi=\bfc^\pi+\gamma(P^\pi)^T\bfv^\pi\quad\mbox{or}\quad \bfv^\pi=(I-\gamma P^\pi)^{-T}\bfc^\pi.
\end{equation*}
\end{lemma}

\begin{defi}[{\bf Optimal counter strategy}]
Let $\tau$ be a strategy of Player 2. A strategy $\sigma$ for Player 1 is an optimal counter strategy against $\tau$ if and only if $\bfv^{(\sigma,\tau)}\leq\bfv^{(\sigma',\tau)}$ for every possible strategy $\sigma'$ of Player 1. Similarly, a strategy $\tau$ for Player 2 is an optimal counter strategy against $\sigma$ if and only if $\bfv^{(\sigma,\tau')}\leq\bfv^{(\sigma,\tau)}$ for every possible strategy $\tau'$ of Player 2. For any given strategy $\sigma$ of Player 1, let $\pi^2(\sigma)$ denote the optimal counter strategy of Player 2. Similarly, for any given strategy $\tau$ of Player 2, let $\pi^1(\tau)$ denote the optimal counter strategy of Player 1.
\end{defi}

\begin{defi}[{\bf Optimal strategy profile}]
A strategy profile $\pi=(\sigma,\tau)$ is optimal if and only if $\sigma$ is an optimal counter strategy against $\tau$ and $\tau$ is an optimal counter strategy against $\sigma$. In this case, $\sigma$ is an optimal strategy for Player 1 and $\tau$ is an optimal strategy for Player 2.
\end{defi}

In addition to the value vector, a strategy profile induces an associated flux vector $\bfx^\pi$, which plays a critical role in our analysis. The flux can be interpreted as a ``discounted flow.'' Suppose that we start with one unit of mass at every state and then run the Markov chain. At each time step, we remove a fraction $1-\gamma$ of the mass from each state and redistribute the remaining mass according to the strategy $\pi$. Summing over all time steps, the total amount of mass passing through each action is its flux. More formally, we have the following lemma (see \cite{post2015simplex}).

\begin{lemma}\label{x-old}
Let $\pi$ be a strategy profile and let $P^\pi$ be the $n$ by $n$ matrix where $P^\pi_{s',s}$ is the probability of transitioning from $s$ to $s'$ using the actions in $\pi$. The action-indexed flux vector $\bfx^\pi$ has the following properties: if action $a$ is not taken in the strategy profile, then $x_a = 0$; if $\pi$ takes $a$ in state $s$, then the nonzero state-indexed flux vector $\bfx^\pi_+$ satisfies 
\[
\bfx^\pi_+=(I-\gamma P^\pi)^{-1}\bfe,
\]
where $\bfe$ is the vector of all ones in dimension $n$. The flux represents the total discounted number of times each action is used when the TBSG starts from all states and evolves according to the Markov chain $P^\pi$ with discounting at each iteration.
\end{lemma}

A flux vector $\bfx^\pi$ can be viewed as a basic feasible solution to the following minimax ``linear programming (LP)'' problem:
\begin{equation}\label{mmax}
\begin{aligned}
\min\max&\quad \bfc^T\bfx\\
\mbox{subject to}&\quad A\bfx=\bfe,\\
&\quad \bfx\geq\bfzero,
\end{aligned}
\quad\mbox{or}\quad
\begin{aligned}
\min\max&\quad (\bfc^1)^T\bfx^1+(\bfc^2)^T\bfx^2\\
\mbox{subject to}&\quad A^1\bfx^1+A^2\bfx^2=\bfe,\\
&\quad \bfx\geq\bfzero,
\end{aligned}
\end{equation}
where $\bfc\in\rr^m$ and $A\in\rr^{n\times m}$. The $a$-th entry of $\bfc$ is $c_a$, and the $a$-th column of $A$ is $\bfe_s-\gamma \bfp_a$ where $a\in\mcA_s$, that is, action $a$ belongs to state $s$. Here, $\bfe_s$ is the vector whose $s$-th entry is one and whose other entries are zero. The flux vector $\bfx\in\rr^m$ is partitioned as $\bfx=(\bfx^1;\bfx^2)$, where $\bfx^1=(x_1;\dots;x_{m_1})$ and $\bfx^2$ contains the remaining components; similarly, $\bfc=(\bfc^1;\bfc^2)$ and $A=[A^1, A^2]$. The minimax problem chooses $\bfx^1$ to minimize and $\bfx^2$ to maximize the linear objective $\bfc^T\bfx=(\bfc^1)^T\bfx^1+(\bfx^2)^T\bfx^2$ subject to the joint constraints $A^1\bfx^1+A^2\bfx^2=\bfe$ and $\bfx\geq\bfzero$.

Under the above ``LP-like'' representation, a flux vector $\bfx$ is optimal, or equivalently corresponds to an optimal strategy profile, if and only if, for a value vector $\bfv$,
\begin{equation}\label{opt0}
(\bfr^1)^T\bfx^1=(\bfr^2)^T\bfx^2=0,\quad\bfr=\bfc-A^T\bfv,\quad\bfr^1\geq\bfzero,\quad\mbox{and}\quad\bfr^2\leq\bfzero,
\end{equation}
where $\bfr$ is the reduced-cost vector associated with the value vector $\bfv$, $\bfr^1=\bfc^1-(A^1)^T\bfv$, and $\bfr^2=\bfc^2-(A^2)^T\bfv$. 

\subsection{Literature and background}
Stochastic games were introduced by Shapley as a dynamic extension of matrix games in which the state evolves according to actions chosen by the players \cite{shapley1953stochastic}. The model has since developed along several related lines, including recursive games and undiscounted stochastic games \cite{gillette1957stochastic}, as well as the broader theory of competitive Markov decision processes \cite{filar1997competitive}. Markov decision processes form the one-player specialization of this framework and provide the classical algorithmic baseline through dynamic programming, policy iteration, and linear programming \cite{howard1960dynamic,puterman1994markov}.

The computational complexity of MDPs has been studied from several complementary viewpoints. Papadimitriou and Tsitsiklis analyzed the complexity of finite-horizon, discounted, and average-cost MDPs \cite{papadimitriou1987complexity}, while Littman, Dean, and Kaelbling surveyed complexity issues surrounding MDP solution methods in planning and artificial intelligence \cite{littman1995complexity}. For discounted MDPs, the complexity landscape is governed by two closely related distinctions: polynomial versus strongly polynomial complexity, and a fixed discount factor versus one supplied as part of the input. Although finite discounted MDPs are polynomial-time solvable through linear programming, strong polynomiality requires a bound on the number of arithmetic operations that is independent of the numerical encoding, including how close $\gamma$ is to one. Ye obtained a strongly polynomial algorithm for discounted MDPs with a fixed discount factor using an interior-point method \cite{ye2005new}, and subsequently established strongly polynomial iteration bounds for the simplex and policy-iteration methods in the same setting \cite{ye2011simplex}. The iteration complexity of policy-improvement methods is nevertheless more subtle: Mansour and Singh provided early worst-case upper bounds for policy iteration \cite{mansour1999complexity}, whereas Melekopoglou and Condon, Fearnley, and subsequent work established exponential lower bounds for several policy-improvement variants \cite{melekopoglou1994complexity,fearnley2010exponential}. Hollanders, Delvenne, and Jungers further demonstrated exponential behavior for discounted MDPs when the discount factor is part of the input \cite{hollanders2012complexity}. Stronger results become possible under additional structural restrictions. For deterministic MDPs, Post and Ye proved that the simplex method with the most-negative-reduced-cost rule is strongly polynomial even with an arbitrary uniform discount factor, obtaining an $O(n^3m^2\log^2 n)$ iteration bound \cite{post2015simplex}. More generally, Scherrer’s structural bounds for policy iteration show that restrictions on transient and recurrent behavior can eliminate explicit dependence on the discount factor for broader classes of MDPs \cite{scherrer2013improved}. These results provide an important precedent for the present paper: recurrence need not obstruct strong polynomiality when the underlying transition structure is sufficiently controlled.

The two-player setting is considerably more delicate. Hansen, Miltersen, and Zwick proved that strategy iteration is strongly polynomial for turn-based stochastic games when the discount factor is fixed, although their iteration bound depends explicitly on $1/(1-\gamma)$ and therefore does not yield strong polynomiality when $\gamma$ is part of the input \cite{hansen2013strategy}. Jia, Wen, and Ye subsequently developed the simplex strategy-iteration framework adopted in this paper, establishing geometric convergence and corresponding complexity guarantees for general TBSGs in the fixed-discount regime \cite{jia2020towards}. Complementary exact and complexity-theoretic analyses further illuminate the difficulty of solving stochastic games on graphs \cite{andersson2009complexity,hansen2011exact}, while exponential lower bounds for strategy-improvement procedures demonstrate that local pivoting rules can perform poorly on unrestricted two-player game classes \cite{friedmann2011exponential}. Thus, when the discount factor is part of the input, strongly polynomial algorithms are not known for broad classes of TBSGs or TBDGs. Against this backdrop, we ask whether strong polynomiality can be recovered by imposing a natural restriction on the recurrent structure of the game graph.

Motivated by this complexity landscape, we study a special class of turn-based deterministic games, called \emph{Turn-Based Deterministic Forward Games} (TBDFGs). In a TBDFG, the underlying directed game graph contains no directed cycle involving actions from both players. This forward-progression condition does not require the graph to be acyclic: recurrent behavior is allowed, but only within states controlled by a single player. Consequently, TBDFGs form a natural intermediate class between acyclic deterministic games and general cyclic turn-based deterministic games. They are substantially more general than finite-horizon or tree-structured models, while still excluding mixed-player feedback loops in which both players can repeatedly influence the same directed cycle. This balance makes TBDFGs both practically relevant and mathematically tractable.

Board games provide important motivating examples as a subclass of TBDFG. The game is divided into finitely many stages,  where each stage consists of states of one of the two players and the immediate following stage consists of the states of the other player. The game is forward in the sense that no state-action transition can return to a state in an earlier stage. At the same time, the model does not require each stage to be acyclic: states controlled by the same player in the same stage may still transition among one another and form local cycles. These self-contained cycles represent local interactions within the decision structure of a single player, while the absence of backward transitions across stages rules out cyclic strategic feedback between the two players. Termination states can be incorporated as absorbing states, namely, single-edge self-loops with arbitrary terminal rewards or costs; see Figure \ref{fig:ab}.


\begin{figure}[H]
\centering

\begin{tikzpicture}[
scale=0.78,
state1/.style={circle, draw=blue!60, fill=blue!20, thick, minimum size=5mm, inner sep=0pt, font=\scriptsize},
state2/.style={circle, draw=red!60, fill=red!20, thick, minimum size=5mm, inner sep=0pt, font=\scriptsize},
legendstate1/.style={circle, draw=blue!60, fill=blue!20, thick, minimum size=3.5mm, inner sep=0pt},
legendstate2/.style={circle, draw=red!60, fill=red!20, thick, minimum size=3.5mm, inner sep=0pt},
caption/.style={font=\scriptsize},
>=stealth,
font=\small
]


\node[legendstate1] at (11.15,0.45) {};
\node[anchor=west, font=\scriptsize] at (11.45,0.45) {Player 1 state $(\mcS^1)$};

\node[legendstate2] at (11.15,-0.05) {};
\node[anchor=west, font=\scriptsize] at (11.45,-0.05) {Player 2 state $(\mcS^2)$};

\draw[->, thick] (10.95,-0.55) -- (11.35,-0.55);
\node[anchor=west, font=\scriptsize] at (11.45,-0.55) {Action};

\draw[dashed, ->, thick]
    (10.95,-1.15) arc[start angle=220,end angle=-80,radius=0.18];
\node[anchor=west, font=\scriptsize] at (11.55,-1.15) {Self-contained directed cycle};


\node[state1] (a1) at (0,0) {1};
\node[state2] (a2) at (-1.5,-1.5) {2};
\node[state2] (a3) at (1.5,-1.5) {3};

\node[state1] (a4) at (-2.5,-3) {4};
\node[state1] (a5) at (-1.5,-3) {5};
\node[state1] (a6) at (-0.5,-3) {6};

\node[state1] (a7) at (0.5,-3) {7};
\node[state1] (a8) at (1.5,-3) {8};
\node[state1] (a9) at (2.5,-3) {9};

\draw[->] (a1)--(a2);
\draw[->] (a1)--(a3);

\draw[->] (a2)--(a4);
\draw[->] (a2)--(a5);
\draw[->] (a2)--(a6);

\draw[->] (a3)--(a7);
\draw[->] (a3)--(a8);
\draw[->] (a3)--(a9);

\node[state2] (a10) at (-2.5,-4.5) {10};
\node[state2] (a11) at (-1.5,-4.5) {11};
\node[state2] (a12) at (-0.5,-4.5) {12};
\node[state2] (a13) at (0.5,-4.5) {13};
\node[state2] (a14) at (1.5,-4.5) {14};
\node[state2] (a15) at (2.5,-4.5) {15};

\draw[->] (a4)--(a10);
\draw[->] (a5)--(a11);
\draw[->] (a6)--(a12);
\draw[->] (a7)--(a13);
\draw[->] (a8)--(a14);
\draw[->] (a9)--(a15);

\draw[->] (a4)--(a11);
\draw[->] (a6)--(a11);
\draw[->] (a6)--(a13);
\draw[->] (a8)--(a13);
\draw[->] (a8)--(a15);

\node[font=\small] at (0,-5.05) {$\vdots$};
\node[caption] at (0,-6.1) {(a) Acyclic game.};


\node[state1] (b1) at (8,0) {1};
\node[state2] (b2) at (6.5,-1.5) {2};
\node[state2] (b3) at (9.5,-1.5) {3};

\node[state1] (b4) at (5.8,-3) {4};
\node[state1] (b5) at (6.8,-3) {5};

\node[state1] (b6) at (8.8,-3) {6};
\node[state1] (b7) at (10.2,-3) {7};

\draw[->] (b1)--(b2);
\draw[->] (b1)--(b3);

\draw[->] (b2)--(b4);
\draw[->] (b2)--(b5);

\draw[->] (b3)--(b6);
\draw[->] (b3)--(b7);

\draw[dashed,blue,->,thick]
    (b5) to[out=20,in=340,looseness=8] (b5);

\draw[dashed,blue,->,thick,bend left=25] (b6) to (b7);
\draw[dashed,blue,->,thick,bend left=25] (b7) to (b6);

\node[state2] (b8)  at (5.8,-4.5) {8};
\node[state2] (b9)  at (6.8,-4.5) {9};
\node[state2] (b10) at (8.8,-4.5) {10};
\node[state2] (b11) at (10.2,-4.5) {11};

\draw[->] (b4)--(b8);
\draw[->] (b5)--(b9);
\draw[->] (b6)--(b10);
\draw[->] (b7)--(b11);
\draw[->] (b4)--(b9);
\draw[->] (b5)--(b8);
\draw[->] (b6)--(b9);
\draw[->] (b6)--(b11);
\draw[->] (b7)--(b10);

\draw[dashed,red,->,thick]
    (b9) to[out=20,in=340,looseness=8] (b9);
\draw[dashed,red,->,thick,bend left=20] (b10) to (b11);
\draw[dashed,red,->,thick,bend left=20] (b11) to (b10);

\node[font=\small] at (8,-5.05) {$\vdots$};
\node[caption] at (8,-6.1) {(b) Two-player staged game.};

\path[use as bounding box] (-5.0,-6.6) rectangle (13.8,2.3);

\end{tikzpicture}

\caption{Special cases of TBDFG.}
\label{fig:ab}

\end{figure}
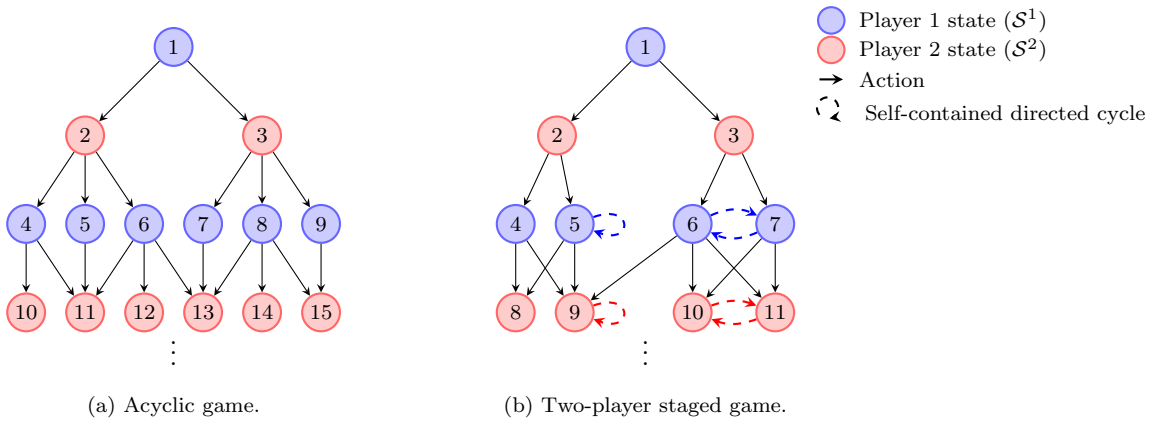

This staged interpretation appears naturally in progression-based games and decision processes, such as Chess and Go. In these games, stages can be viewed as successive decision steps from the initial position to termination. The game progresses forward because each move advances the play to a later decision stage, while terminal outcomes such as checkmate, resignation, draw, or scoring can be represented by absorbing termination states. In Go, for instance, a board position that would repeat an earlier position is prohibited by the ko or superko rule. Equivalently, such a prohibited repetition can be modeled as reaching a terminal invalid-move state, rather than as a transition back to an earlier stage. The simplest such instance in Go is shown in Figure \ref{fig:simple-ko-loop} where Black plays at $x$ and captures the white stone at $y$, producing position $T$. If White were allowed to immediately play at $y$, White would capture the black stone at $x$ and return the board to position $S$. The ko rule prevents this infinite two-player cycle. Under this representation, the game has no mixed-player interaction cycle: repetitions are absorbed by terminal states, and the remaining legal transitions progress forward through the staged game tree. Thus, Chess and Go illustrate the practical relevance of the TBDFG framework, whose essential structure is forward progression together with possible local self-contained cycles.

\begin{figure}[H]
\centering
\begin{tikzpicture}[scale=0.75, every node/.style={font=\small}]

\newcommand{\board}[2]{
    \begin{scope}[shift={#1}]
        \foreach \i in {0,...,4}{
            \draw[gray!60] (0,\i) -- (4,\i);
            \draw[gray!60] (\i,0) -- (\i,4);
        }

        #2
    \end{scope}
}

\newcommand{\black}[2]{
    \fill[black] (#1,#2) circle (0.16);
}
\newcommand{\white}[2]{
    \fill[white] (#1,#2) circle (0.16);
    \draw[black] (#1,#2) circle (0.16);
}
\newcommand{\markpt}[2]{
    \draw[red, thick] (#1,#2) circle (0.24);
}

\board{(0,0)}{
    \white{1}{2}
    \white{2}{1}
    \white{2}{3}
    \white{3}{2}

    \black{3}{1}
    \black{3}{3}
    \black{4}{2}

    \markpt{2}{2}
}
\node at (2,-0.7) {State $S$: Black to play};
\node[red] at (2,4.45) {$x$};

\draw[->, thick] (4.8,2) -- (6.2,2);
\node at (5.5,2.45) {$B:x$};

\board{(7,0)}{
    \white{1}{2}
    \white{2}{1}
    \white{2}{3}

    \black{2}{2}
    \black{3}{1}
    \black{3}{3}
    \black{4}{2}

    \markpt{3}{2}
}
\node at (9,-0.7) {State $T$: White to play};
\node[red] at (10,4.45) {$y$};

\draw[->, thick] (11.8,2) -- (13.2,2);
\node at (12.5,2.45) {$W:y$};

\board{(14,0)}{
    \white{1}{2}
    \white{2}{1}
    \white{2}{3}
    \white{3}{2}

    \black{3}{1}
    \black{3}{3}
    \black{4}{2}

    \markpt{2}{2}
}
\node at (16,-0.7) {Back to $S$};
\node[red] at (16,4.45) {$x$};

\draw[->, thick, dashed] (16,5) .. controls (10,6.3) and (6,6.3) .. (2,5);
\node at (9,6.35) {without ko: $S \to T \to S \to T \to \cdots$};

\end{tikzpicture}
\caption{A simplest ko instance in Go.}
\label{fig:simple-ko-loop}
\end{figure}
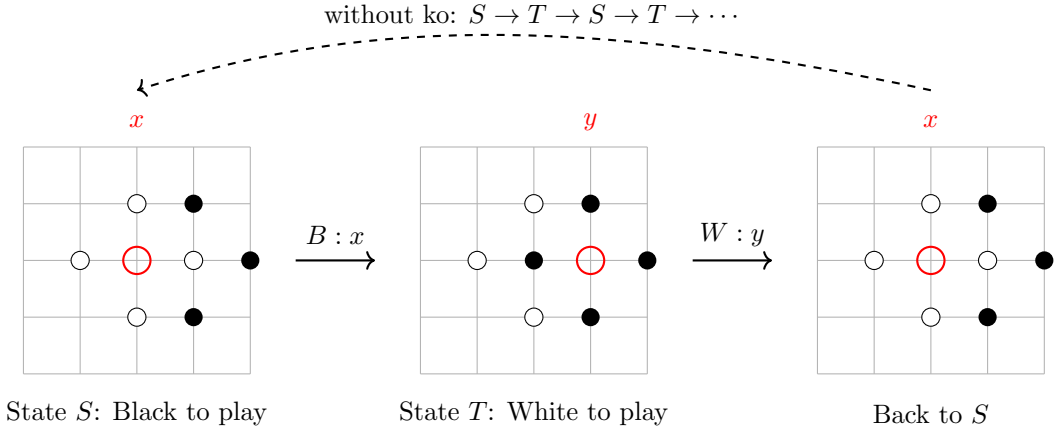

The success of AlphaGo further highlights the algorithmic importance of exploiting structure in large deterministic games. AlphaGo and AlphaGo Zero combine policy improvement, value estimation, and search to make the enormous game tree of Go computationally tractable \cite{silver2016mastering,silver2017mastering}. Although these methods are learning-based and heuristic rather than complexity-theoretic, they demonstrate that structural information and iterative value improvement can be powerful tools for solving large turn-based games. Our work is complementary to this line of research. Instead of designing a practical learning system for Go, we identify a structural subclass of the game that captures the forward-progression principle in Go and other board games, and prove that a simple strategy-iteration method converges in strongly polynomial time for this class.

We make the formal definition of TBDFGs as the following.

\begin{defi}[{\bf TBDFG}]\label{tbdfg}
A \emph{Turn-Based Deterministic Forward Game (TBDFG)} is a \emph{TBDG} satisfying the additional \emph{forward-progression} property that the underlying game graph contains no directed cycle involving actions from both players.
\end{defi}


The forward game setting is naturally suited to broader classes of games and decision processes with hierarchical, staged, or predominantly forward dynamics. In many such models, the state space can be organized by a progress variable, such as time, level, resource depletion, accumulated information, or position in a task hierarchy. Most actions advance the system with respect to this variable, while local loops may occur within a stage when a player repeats a decision, retries an action, waits, refines a subtask, or cycles among internal configurations before the game proceeds. When these local cycles are controlled by a single player, the resulting game satisfies the TBDFG condition. This setting captures deterministic planning problems with adversarial choices, hierarchical game trees with local repetitions, verification and synthesis games with layered progress measures, and restricted graph games in which cycles are confined to one controller.

The assumption is also motivated by graph games, where the transition graph plays a central role in both modeling and complexity analysis. Simple stochastic games~\cite{condon1992complexity,ludwig1995subexponential}, mean-payoff games~\cite{ehrenfeucht1979positional,zwick1996complexity}, and parity games~\cite{emerson1991tree,calude2017deciding} demonstrate how graph structure can strongly influence algorithmic behavior. Since cycles are often the main source of long-run strategic complexity, the forward-game condition isolates this difficulty by allowing only player-internal cycles. Consequently, TBDFGs preserve nontrivial recurrent behavior while ruling out directed feedback cycles controlled jointly by both players.

These models can have substantially richer structures than special cases such as acyclic games and two-player staged games, illustrated in Figure~\ref{fig:ab}. Figure~\ref{fig:cd} presents examples of more general TBDFGs. Figure~\ref{fig:cd}(a) contains cycles in Player~2's states and does not exhibit the alternating stage-like structure typical of board games. Figure~\ref{fig:cd}(b) shows that multiple single-player cycles may coexist and form more complicated recurrent structures. The defining feature of a TBDFG is that every directed cycle remains internal to one player, allowing rich local recurrence while preserving forward progression in the strategic interaction between the two players.

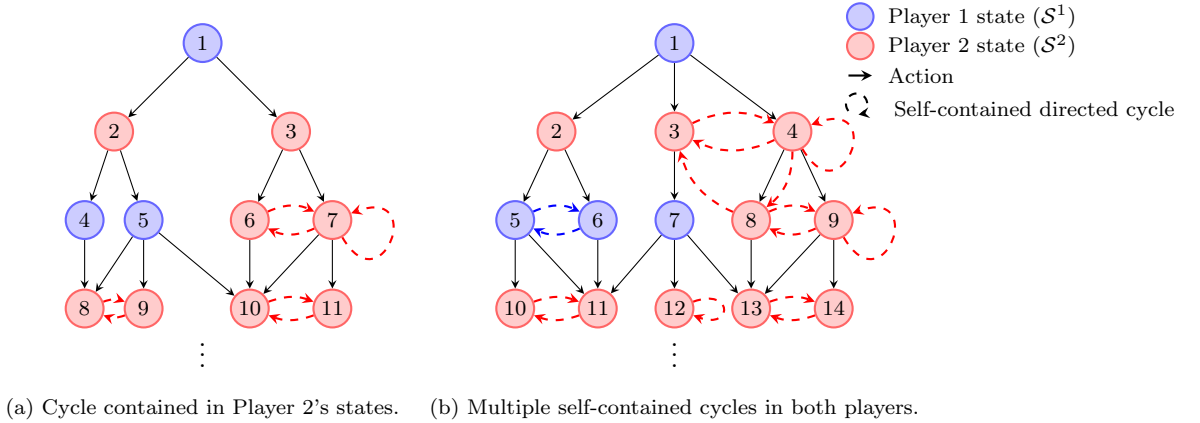
\begin{figure}[H]
\centering

\begin{tikzpicture}[
scale=0.78,
state1/.style={circle, draw=blue!60, fill=blue!20, thick, minimum size=5mm, inner sep=0pt, font=\scriptsize},
state2/.style={circle, draw=red!60, fill=red!20, thick, minimum size=5mm, inner sep=0pt, font=\scriptsize},
legendstate1/.style={circle, draw=blue!60, fill=blue!20, thick, minimum size=3.5mm, inner sep=0pt},
legendstate2/.style={circle, draw=red!60, fill=red!20, thick, minimum size=3.5mm, inner sep=0pt},
caption/.style={font=\scriptsize},
>=stealth,
font=\small
]


\node[legendstate1] at (11.15,0.45) {};
\node[anchor=west, font=\scriptsize] at (11.45,0.45) {Player 1 state $(\mcS^1)$};

\node[legendstate2] at (11.15,-0.05) {};
\node[anchor=west, font=\scriptsize] at (11.45,-0.05) {Player 2 state $(\mcS^2)$};

\draw[->, thick] (10.95,-0.55) -- (11.35,-0.55);
\node[anchor=west, font=\scriptsize] at (11.45,-0.55) {Action};

\draw[dashed, ->, thick]
    (10.95,-1.15) arc[start angle=220,end angle=-80,radius=0.18];
\node[anchor=west, font=\scriptsize] at (11.55,-1.15) {Self-contained directed cycle};


\node[state1] (c1) at (0,0) {1};
\node[state2] (c2) at (-1.5,-1.5) {2};
\node[state2] (c3) at (1.5,-1.5) {3};

\node[state1] (c4) at (-2,-3) {4};
\node[state1] (c5) at (-1,-3) {5};

\node[state2] (c6) at (0.8,-3) {6};
\node[state2] (c7) at (2.2,-3) {7};

\draw[->] (c1)--(c2);
\draw[->] (c1)--(c3);

\draw[->] (c2)--(c4);
\draw[->] (c2)--(c5);

\draw[->] (c3)--(c6);
\draw[->] (c3)--(c7);

\draw[dashed,red,->,thick,bend left=25] (c6) to (c7);
\draw[dashed,red,->,thick,bend left=25] (c7) to (c6);

\draw[dashed,red,->,thick]
    (c7) to[out=300,in=20,looseness=8] (c7);

\node[state2] (c8)  at (-2,-4.5) {8};
\node[state2] (c9)  at (-1,-4.5) {9};
\node[state2] (c10) at (0.8,-4.5) {10};
\node[state2] (c11) at (2.2,-4.5) {11};

\draw[->] (c4)--(c8);
\draw[->] (c5)--(c9);
\draw[->] (c6)--(c10);
\draw[->] (c7)--(c11);

\draw[->] (c5)--(c10);
\draw[->] (c5)--(c8);
\draw[->] (c7)--(c10);

\draw[dashed,red,->,thick,bend left=20] (c8) to (c9);
\draw[dashed,red,->,thick,bend left=20] (c9) to (c8);
\draw[dashed,red,->,thick,bend left=20] (c10) to (c11);
\draw[dashed,red,->,thick,bend left=20] (c11) to (c10);

\node[font=\small] at (0,-5.15) {$\vdots$};
\node[caption] at (0,-6.2) {(a) Cycle contained in Player 2's states.};


\node[state1] (d1) at (8,0) {1};

\node[state2] (d2) at (6,-1.5) {2};
\node[state2] (d3) at (8,-1.5) {3};
\node[state2] (d4) at (10,-1.5) {4};

\node[state1] (d5) at (5.3,-3) {5};
\node[state1] (d6) at (6.7,-3) {6};

\node[state1] (d7) at (8,-3) {7};

\node[state2] (d8) at (9.3,-3) {8};
\node[state2] (d9) at (10.7,-3) {9};

\draw[->] (d1)--(d2);
\draw[->] (d1)--(d3);
\draw[->] (d1)--(d4);

\draw[->] (d2)--(d5);
\draw[->] (d2)--(d6);

\draw[->] (d3)--(d7);

\draw[->] (d4)--(d8);
\draw[->] (d4)--(d9);

\draw[dashed,blue,->,thick,bend left=25] (d5) to (d6);
\draw[dashed,blue,->,thick,bend left=25] (d6) to (d5);

\draw[dashed,red,->,thick,bend left=25] (d8) to (d9);
\draw[dashed,red,->,thick,bend left=25] (d9) to (d8);

\draw[dashed,red,->,thick,bend left=25] (d8) to (d3);
\draw[dashed,red,->,thick,bend left=25] (d3) to (d4);
\draw[dashed,red,->,thick,bend left=25] (d4) to (d3);
\draw[dashed,red,->,thick,bend left=25] (d4) to (d8);

\draw[dashed,red,->,thick]
    (d9) to[out=300,in=20,looseness=8] (d9);

\draw[dashed,red,->,thick]
    (d4) to[out=300,in=20,looseness=8] (d4);

\node[state2] (d10) at (5.3,-4.5) {10};
\node[state2] (d11) at (6.7,-4.5) {11};
\node[state2] (d12) at (8.0,-4.5) {12};
\node[state2] (d13) at (9.3,-4.5) {13};
\node[state2] (d14) at (10.7,-4.5) {14};

\draw[->] (d5)--(d10);
\draw[->] (d6)--(d11);
\draw[->] (d7)--(d12);
\draw[->] (d8)--(d13);
\draw[->] (d9)--(d14);

\draw[->] (d5)--(d11);
\draw[->] (d7)--(d11);
\draw[->] (d7)--(d13);
\draw[->] (d9)--(d13);

\draw[dashed,red,->,thick,bend left=20] (d10) to (d11);
\draw[dashed,red,->,thick,bend left=20] (d11) to (d10);
\draw[dashed,red,->,thick] (d12) to[out=20,in=340,looseness=8] (d12);
\draw[dashed,red,->,thick,bend left=20] (d13) to (d14);
\draw[dashed,red,->,thick,bend left=20] (d14) to (d13);

\node[font=\small] at (8,-5.15) {$\vdots$};
\node[caption] at (8,-6.2) {(b) Multiple self-contained cycles in both players.};

\path[use as bounding box] (-5.0,-6.7) rectangle (13.8,2.3);

\end{tikzpicture}

\caption{Examples of general TBDFG.}
\label{fig:cd}

\end{figure}
%
%
%

This richer model structure also makes the problem challenging to solve. Acyclic games can be solved by backward dynamic programming because the value of each state depends only on later states. In contrast, TBDFGs may contain recurrent components, so a direct backward argument is insufficient. At the same time, the forward property is graph-theoretic and can be exploited directly once it is known. Thus, the main point of our first result is not merely that TBDFGs are strongly polynomially solvable. Rather, it is that a standard simple strategy-iteration method, designed for general turn-based stochastic games and run without testing for the TBDFG structure, nevertheless follows a strongly polynomial trajectory on this subclass. This provides a structure-oblivious convergence guarantee for a generic local-improvement method. Our second result is complementary: when the forward structure is known or certified, one can exploit the SCC decomposition explicitly and obtain
a sharper specialized algorithm. Our contributions are as follows.
\begin{itemize}
\item We formalize the TBDFG, a structural subclass of TBDG in which no directed cycle contains actions from both players. This class is motivated by forward-progression structures in staged games, deterministic planning, and graph games with local recurrence.
\item We analyze the simple strategy-iteration algorithm of \cite{hansen2013strategy,jia2020towards}, which is not tailored to TBDFGs but a general strategy-iteration procedure for solving TBSGs. We prove that even this general algorithm, run without detecting or exploiting the forward structure, converges in strongly polynomial time on TBDFGs, with an iteration bound of $O(n^6m^4 \log^4 n)$.This establishes strongly polynomial solvability of TBDFGs by the general strategy-iteration algorithm. 
\item We then exploit the graph structure of TBDFGs directly and develop a specialized backward strongly connected component (SCC) propagation algorithm. This second algorithm uses information that Algorithm \ref{alg} ignores for generality: the decomposition of the game graph into single-player SCCs ordered acyclically. It therefore gives a sharper structure-aware bound of $O(n^3m^2 \log^2 n)$.
\end{itemize}

The rest of this paper is organized as follows. In Subsection \ref{sec:prelim}, we introduce notation, terminology, and known properties for TBDGs. In Section \ref{sec:alg}, we study the general simple strategy-iteration algorithm and prove that it is strongly polynomial for TBDFGs. In Section~\ref{sec:specialized}, we use the graph structure of TBDFGs to design a specialized backward SCC propagation algorithm with improved complexity. Section~\ref{sec:summary} summarizes our results.

%


\subsection{Preliminaries}\label{sec:prelim}
We use the following notation throughout the paper. Let $\rr^n$ denote the Euclidean space of dimension $n$, and let $\mathbb{N}$ denote the set of all natural numbers. For any $n\in\mathbb{N}\setminus\{0\}$, let $[n]$ denote the set $\{1,\dots,n\}$. Let $\bfe$ and $\bfzero$ be the all-ones and all-zeros vectors, respectively, with dimensions clear from the context. To simplify the discussion, we use the following equivalent scaled ``LP-like'' representation.
\begin{equation}\label{prob}
\begin{aligned}
\min\max&\quad \bfc^T\bfx\\
\mbox{subject to}&\quad A\bfx=(1-\gamma)\bfe:=\bfb,\\
&\quad \bfx\geq\bfzero,
\end{aligned}
\quad\mbox{or}\quad
\begin{aligned}
\min\max&\quad (\bfc^1)^T\bfx^1+(\bfc^2)^T\bfx^2\\
\mbox{subject to}&\quad A^1\bfx^1+A^2\bfx^2=\bfb,\\
&\quad \bfx\geq\bfzero.
\end{aligned}
\end{equation}
The formulation in \eqref{prob} is equivalent to \eqref{mmax} up to the scaling factor $(1-\gamma)$. This notation simplifies the discussion below, while all previous results, namely Lemmas \ref{lemma1} and \ref{x-old} and \eqref{opt0}, continue to hold after replacing $\bfe$ with $\bfb$. The following results follow from these and  by applying the $(1-\gamma)$ scaling to the proofs in \cite{hansen2013strategy,shapley1953stochastic}.

\begin{prop}\label{prop-basic}
 For any TBDG given by $\mcG=(\mcS,\mcA,\bfc,\bfP,\gamma)$ and any strategy profile $\pi$ of the game, the following properties hold.
\begin{enumerate}[label=(\roman*)]
\item
The state-indexed cost-to-go value vector $\bfv^\pi$ satisfies
\[
\bfv^\pi=\bfc^\pi+\gamma(P^\pi)^T\bfv^\pi\quad\mbox{or}\quad \bfv^\pi=(I-\gamma P^\pi)^{-T}\bfc^\pi,
\]
where $P^\pi$ is the $n$ by $n$ matrix where $P^\pi_{s,s'}$ is the probability of transitioning from $s'$ to $s$ using the actions in $\pi$ and $\bfc^\pi$ be the vector of costs for each state under $\pi$. 
\item
The associated (scaled) flux vector $\bfx^\pi$ satisfies $\bfx^\pi_a=0$ for all $a\notin\pi$, and
\[
\bfx^\pi_+=(I-\gamma P^\pi)^{-1}\bfb
\label{def-x-new}
\]
where $\bfb=(1-\gamma)\bfe$ and $\bfx^\pi_+$ is the nonzero state-index flux vector

\item
The feasible polytope of \eqref{prob} is nondegenerate for $0<\gamma<1$, and there is a bijection between vertices and strategy profiles of the TBDG.

\item For the (scaled) flux vector, we have
\[
\sum_{a\in\pi}\bfx^\pi_a=n,
\quad
\bfx^\pi_a=0 \quad \forall a\notin\pi.
\]
Furthermore, we have $1-\gamma\leq x^\pi_a\leq(1-\gamma)n$ if $a\in\pi$ is a path-action and $1\leq x^\pi_a\leq n$ if $a\in\pi$ is a cycle-action.
\item
The reduced-cost vector is $\bfr^\pi=\bfc-A^T\bfv^\pi$ where $A=I-\gamma P$. Moreover, we have
\[
(\bfr^\pi)^T\bfx^\pi=0
\qquad\text{and}\qquad
\bfc^T\bfx^\pi=(\bfc^\pi)^T\bfx^\pi_+=\bfb^T\bfv^\pi.
\label{rx}
\]
Furthermore, for any two strategy profiles $\pi,\pi'$, we have
\[
\bfb^T(\bfv^{\pi'}-\bfv^\pi)=(\bfr^\pi)^T\bfx^{\pi'}=(\bfr^\pi_{\pi'})^T\bfx^{\pi'}_+.
\]
where $\bfr^\pi_{\pi'}$ is the truncation of $\bfr^\pi$ on the set of actions used in $\pi'$.
\item In particular, let $\pi^*=(\pi_*^1,\pi_*^2)$ be the optimal strategy profile and denote the objective value by $\nu^*$. Then we have
\begin{equation}\label{vstar}
\nu^*:=\bfc^T\bfx^{\pi_*}=(\bfc^{\pi_*})^Tx^{\pi_*}_+=\bfb^T\bfv^{\pi_*}\quad\mbox{and}\quad (\bfr^{\pi_*})^T\bfx^{\pi_*}=0.
\end{equation}
\item A flux vector $\bfx$ is optimal, or equivalently, corresponds to an optimal strategy profile, if and only if, for a value vector $\bfv$,
\begin{equation}\label{opt}
(\bfr^1)^T\bfx^1=(\bfr^2)^T\bfx^2=0,\quad\bfr=\bfc-A^T\bfv,\quad\bfr^1\geq\bfzero,\quad\mbox{and}\quad\bfr^2\leq\bfzero,
\end{equation}
where $\bfr$ is the reduced-cost vector associated with the value vector $\bfv$, $\bfr^1=\bfc^1-(A^1)^T\bfv$, and $\bfr^2=\bfc^2-(A^2)^T\bfv$.
\end{enumerate}

\end{prop}

For any fixed strategy of either player, finding the other player's optimal counter strategy in a TBDG reduces to a deterministic MDP. The following result is a corollary of \cite[Theorem 3.7]{post2015simplex}.

\begin{prop}\label{thm-MDP}
For a TBDG, using the simplex method with the simple pivoting rule, for any fixed strategy $\sigma$ of Player 1, the optimal counter strategy $\pi^2(\sigma)$ can be found in $O(n^3m^2\log^2n)$ pivot steps of simplex method; and for any fixed strategy $\tau$ of Player 2, the counter optimal strategy $\pi^1(\tau)$ can be found in $O(n^3m^2\log^2n)$ pivot steps of simplex method.
\end{prop}

\section{Simple strategy-iteration algorithm and complexity results}\label{sec:alg}


In this section, we begin by studying a general simple strategy-iteration algorithm for TBSGs, rather than an algorithm tailored specifically to TBDFGs. Its input contains no forward-structure certificate, and its execution does not test for mixed-player cycles or modify its pivot rule when the property is present. The TBDFG assumption enters only in the convergence proof. Accordingly, the main theorem of this section provides a \emph{structure-oblivious} performance guarantee: the same algorithm may be run on a general game, but it has a strongly polynomial complexity whenever the instance belongs to the TBDFG class. 

The strategy-iteration method extends the original policy-iteration method for MDPs \cite{howard1960dynamic}. It updates the action at a state $s\in\mcS^1$ with the most negative reduced cost in $\bfr$ under the strategy profile $\pi=(\sigma,\pi^2(\sigma))$. Specifically, for every state $s\in\mcS^1$, if $\delta_s=\min\{r_a:a\in\mcA_s\}<0$, then the current action in strategy $\sigma$ is switched to $a_s=\arg\min\{r_a:a\in\mcA_s\}$ at state $s$; ties may be broken arbitrarily in each state. The method then repeats with the new strategy profile $\pi^+=(\sigma^+,\pi^2(\sigma^+))$, where possibly multiple actions in $\sigma$ are replaced. Each iteration can be completed in strongly polynomial time (see Proposition \ref{thm-MDP}). The difference between the standard strategy-iteration algorithm and the simple strategy-iteration algorithm is that the latter replaces the action in only one state $s\in\mcS^1$, namely, the state with the smallest negative reduced cost.

We now state the simple strategy-iteration algorithm, which coincides with Algorithm 1 in \cite{jia2020towards}. In each iteration, Algorithm \ref{alg} updates Player 1's strategy through a simple pivoting step, with ties broken arbitrarily, and then finds Player 2's optimal strategy against the updated Player 1 strategy using simplex method with simple pivoting rule.  The algorithm can be described as follows.

\begin{algorithm}[H]
\caption{Simple strategy-iteration algorithm}\label{alg}
\begin{algorithmic}[1]
\REQUIRE Initialize $\pi_0=(\sigma_0,\tau_0)$, where $\tau_0=\pi^2(\sigma_0)$ and $\bfv^{\pi_0}$ are as defined in Proposition \ref{prop-basic}. Set $k=1$.
\STATE Run one simplex step for Player 1, and denote the updated Player 1 strategy by $\sigma_k$; namely, pivot in the action with the smallest negative modified cost for Player 1.
\STATE Call the Simplex method with the simple pivoting rule in \cite[Section 2]{post2015simplex} with initial strategy $\tau_{k-1}$ to find the optimal strategy for Player 2, and set $\tau_k=\pi^2(\sigma_k)$.
\STATE Update the value function $\bfv^{\pi_k}$, where $\pi_k=(\sigma_k,\tau_k)$.
\STATE If $\bfv^{\pi_k}\neq\bfv^{\pi_{k-1}}$, set $k\leftarrow k+1$ and go to Step 1. Otherwise, terminate the algorithm and return $\pi_k=(\sigma_k,\tau_k)$.
\end{algorithmic}
\end{algorithm}

\begin{rem}
\begin{itemize}
\item In Step 1 of Algorithm~\ref{alg}, Player 1's strategy is not updated if there is no action with negative modified cost. In particular, if $\pi_0$ is already an optimal strategy profile, then Step 1 leaves the strategy profile unchanged, and the algorithm terminates without performing any update.
\item By Proposition \ref{prop-basic}(iii), the feasible polytope in \eqref{prob} is nondegenerate, and its vertices are in one-to-one correspondence with strategy profiles of the TBDG. Hence the simplex pivot in Step 1 of Algorithm \ref{alg} is well defined as a move between strategy profiles. Moreover, by Proposition \ref{thm-MDP}, once Player 1's strategy $\sigma_k$ is fixed, Player 2's optimal counter strategy $\tau_k=\pi^2(\sigma_k)$ in Step 2 can be computed by the simplex method with the simple pivoting rule in strongly polynomial iteration time. Therefore each iteration of Algorithm \ref{alg} is well defined as a strategy-iteration procedure for TBDGs.
\end{itemize}
\end{rem}

It is useful to distinguish the algorithmic template from the complexity analysis. Algorithm~\ref{alg} is the same simple strategy-iteration scheme studied for general turn-based stochastic games in \cite{hansen2013strategy,jia2020towards}; the novelty here is not a new pivot rule, but a stronger guarantee for the TBDFG subclass. In particular, the bound proved below is strongly polynomial when the uniform discount factor is part of the input, and it contains no dependence on $1/(1-\gamma)$. The proof is inspired in part by the reduced-cost and flux analysis of the simplex method for deterministic MDPs in \cite{post2015simplex}, but the two-player setting creates a difficulty that is absent from a one-player minimization problem. After each Player~1 pivot, Player~2 is reoptimized, so the resulting sequence is governed by a minimax interaction rather than by a single monotone simplex trajectory. The TBDFG condition rules out directed cycles that contain actions of both players, thereby excluding mixed-player recurrent feedback. Nevertheless, Player~1 and Player~2 actions may still alternate along transient paths before the play enters a one-player recurrent component. This transient interaction is the main new obstacle in the proof: unlike the one-player simplex analysis of \cite{post2015simplex}, progress cannot be measured solely by the pivot sequence of Player~1, since each such pivot may change Player~2's optimal response; the analysis below must therefore track the cumulative effects of these counter-optimizations along mixed-player paths, which requires new estimates beyond those used in the MDP setting in \cite{post2015simplex}.

We next prove that Algorithm~\ref{alg} computes optimal strategies for both players in time polynomial in $m$, $n$, and $\log n$, as stated in Theorem~\ref{thm}. We begin with a property of optimal counter strategies that holds for all TBDGs.


\begin{lemma}\label{l-opt}
Denote by $\nu^1$ the objective value of any strategy profile $(\sigma,\pi^2(\sigma))$, and denote by $\nu^2$ the objective value of any strategy profile $(\pi^1(\tau),\tau)$. Then $\nu^2\leq \nu^*\leq\nu^1$. Moreover, the corresponding reduced-cost vector satisfies
\[
r^{(\sigma,\pi^2(\sigma))}_a\leq0\quad\forall a\in \mcA^2\quad\mbox{and}\quad r^{(\pi^1(\tau),\tau)}_a\geq0\quad\forall a\in \mcA^1.
\]
and $r^\pi_a=0$ for all actions taken in the strategy profile.
\end{lemma}

\begin{proof}
For any strategy $\sigma$ of Player 1, let the strategy profile be $\pi=(\sigma,\pi^2_*)$. Then we have
\[
\bfc^T\bfx^\pi-\nu^*=(\bfr^{\pi^*})^T\bfx^\pi=\sum_{a\in\mcA^1}r^{\pi^*}_ax^\pi_a+\sum_{a\in \mcA^2}r^{\pi^*}_ax^\pi_a\geq0,
\]
where the first equality follows from Proposition \ref{prop-basic}(v) and \eqref{vstar}, and the inequality follows from $\bfx^\pi\geq\bfzero$, $r^{\pi^*}_a\geq0$ for any $a\in \mcA^1$, $r^{\pi^*}_a=0$ for any $a\in\pi^2_*$, and $\bfx^\pi_a=0$ for any $a\in \mcA^2/\pi^2_*$. On the other hand, $\nu^1$ is the objective value of the strategy profile $(\sigma,\pi^2(\sigma))$, where $\pi^2(\sigma)$ is the optimal counter strategy of Player 2 against $\sigma$ and maximizes the objective value. Thus $\nu^1$ should be at least the value of the strategy profile $\pi=(\sigma,\pi^2_*)$, namely $\nu^1\geq\nu^*$. Similarly, one can prove $\nu^*\geq\nu^2$.

Finally, the reduced-cost sign property follows directly from the counter-optimality of the other player, since the resulting problem is an MDP.
\end{proof}



We next prove that the value vector decreases monotonically throughout the iterations. The proof is inspired by \cite[Lemma 5.1]{hansen2013strategy} and holds for TBDGs.

\begin{lemma}\label{lemma9}
The value function obtained by Algorithm \ref{alg} decreases strictly at each iteration, i.e., $\bfv^{\pi_{k-1}}\geq\bfv^{\pi_k}$ and $\bfb^T\bfv^{\pi_{k-1}}>\bfb^T\bfv^{\pi_k}$, unless Algorithm \ref{alg} terminates.
\end{lemma}

\begin{proof}
Without loss of generality, we only prove the statement for the first iteration.

First, we show the monotonicity of the value vector. By Lemma \ref{l-opt} and $\tau_0=\pi^2(\sigma_0)$ (see Input of Algorithm \ref{alg}), we have $r^{\pi_0}_a\leq0$ for all $a\in\mcA^2$ and $r^{\pi_0}_a=0$ for all $a\in\pi_0$. Moreover, according to the updating rule used in Step 1 of Algorithm \ref{alg}, we have $r^{\pi_0}_a=\min\{r^{\pi_0}_{a'}:a'\in\mcA^1\}$ for $a\in\sigma_1\setminus \sigma_0$. This implies that $r^{\pi_0}_a\leq0$ for all $a\in\pi_1$. Using this together with Proposition \ref{prop-basic}(i) and (v), we have
\[
\bfv^{\pi_1}-\bfv^{\pi_0}=(I-\gamma P^{\pi_1})^{-T}\left(\bfc^{\pi_1}-(I-\gamma P^{\pi_1})^T\bfv^{\pi_0}\right)=(I-\gamma P^{\pi_1})^{-T}\bfr^{\pi_0}_{\pi_1}\leq0
\]
where the first equality follows from Proposition \ref{prop-basic}(i), the second equality follows from the definition of $\bfr$ in Proposition \ref{prop-basic}(v), and the inequality follows from $r^{\pi_0}_a\leq0$ for all $a\in\sigma_1$ and the fact that $P^{\pi_1}$ is the probability transition matrix. Consequently, the obtained value function decreases.

%


Next, we show that $\bfb^T\bfv^{\pi_0}>\bfb^T\bfv^{\pi_1}$ unless Algorithm \ref{alg} terminates. If not, we have $\min\{r^{\pi_0}_a:a\in\mcA^1\}<0$. Denoting $a'=\argmin\{r^{\pi_0}_a:a\in\mcA^1\}$, we have
\[
\bfb^T\bfv^{\pi_0}-\bfb^T\bfv^{\pi_1}=-(\bfr^{\pi_0}_{\pi_1})^T\bfx^{\pi_1}_+=-\sum_{a\in\sigma_1}r^{\pi_0}_ax^{\pi_1}_a-\sum_{s\in\tau_1}r^{\pi_0}_ax^{\pi_1}_a\geq-\sum_{a\in\sigma_1}r^{\pi_0}_a\bfx^{\pi_1}_a=-r^{\pi_0}_{a'}x^{\pi_1}_{a'}>0
\]
where the first equality follows from Proposition \ref{prop-basic}(v), the first inequality follows from $\tau_0=\pi^2(\sigma_0)$ which implies that $r^{\pi_0}_a\leq0$ for all $a\in\tau_1$, the last equality follows from the updating scheme in Step 1 of Algorithm \ref{alg}, and the last inequality follows from the fact that $r^{\pi_0}_{a'}<0$. Hence the result follows.
\end{proof}

\subsection{Improvement with a new cycle variable}
In this subsection, we show that the objective value $\bfb^T\bfv$ improves significantly when a pivoted-in action of Player 1 becomes a cycle action. We first prove the following objective-reduction lemmas for the case in which a new cycle variable is introduced.

\begin{lemma}\label{lemma10}
Let $\bfv^{\pi^*}$ be given in \eqref{vstar} and $\Delta_k=\min\{r^{\pi_k}_a:a\in\mcA^1\}$ be the minimum modified cost of the strategy profile $\pi_k$ at the beginning of the $(k+1)$-th iteration. Then, we have
\begin{equation}\label{e1}
\bfb^T\bfv^{\pi^*}\geq\bfb^T\bfv^{\pi_k}-n|\Delta_k|.
\end{equation}
\end{lemma}

\begin{proof}
If $\Delta_k\geq0$, then $\bfv^{\pi_k}=\bfv^{\pi^*}$ is the optimal value function, and this inequality holds automatically.

Now, consider the case $\Delta_k<0$. Let $\sigma_k^*=\pi^1(\tau_k)$. Then, it follows from Lemma \ref{l-opt} that
\begin{equation}\label{e2}
\bfb^T\bfv^{(\sigma_k^*,\tau_k)}\leq\bfb^T\bfv^{\pi^*}.
\end{equation}
%
Additionally, based on the definition of $\sigma_k^*$, we have
\begin{equation}\label{e3}
\bfb^T\bfv^{(\sigma_k^*,\tau_k)}\geq\bfb^T\bfv^{(\sigma_k,\tau_k)}-n|\Delta_k|.
\end{equation}
Since $\Delta_k$ has the same interpretation as the modified cost in linear programming, this inequality follows from the simplex-method analysis for MDPs (see \cite{post2015simplex}). Consequently, \eqref{e1} follows by combining \eqref{e2} and \eqref{e3}.
\end{proof}

We next show that a substantial improvement is achieved once a pivoted-in action becomes part of a cycle in a TBDFG.
\begin{lemma}\label{lem:sig}
Suppose the game is a TBDFG. Let $\bfv^{\pi^*}$ be given in \eqref{vstar}. Suppose a Player 1's cycle is created on the $k$-th iteration, then
\[
\bfb^T(\bfv^{\pi_k}-\bfv^{\pi^*})\leq\left(1-\frac{1}{n}\right)\bfb^T(\bfv^{\pi_{k-1}}-\bfv^{\pi^*}).
\]
\end{lemma}


\begin{proof}

Let $a'\in\mcA^{s'}$ be the newly introduced action of Player 1 corresponding to the minimum modified cost during the $k$-th iteration. Since a Player 1 cycle is created during this iteration and Player 2’s response cannot affect this cycle by the forward-progression property, $a'$ is a cycle action of $\pi_k$ at the end of this iteration. Observe that
\begin{align*}
\bfb^T(\bfv^{\pi_k}-\bfv^{\pi^*})=&\bfb^T(\bfv^{\pi_{k-1}}-\bfv^{\pi^*})-\bfb^T(\bfv^{\pi_{k-1}}-\bfv^{\pi_k})\\
=&\bfb^T(\bfv^{\pi_{k-1}}-\bfv^{\pi^*})+\sum_{a\in \pi_k}r^{\pi_{k-1}}_ax^{\pi_k}_a\\
\leq&\bfb^T(\bfv^{\pi_{k-1}}-\bfv^{\pi^*})+\sum_{a\in \sigma_k}r^{\pi_{k-1}}_ax^{\pi_k}_a\\
=&\bfb^T(\bfv^{\pi_{k-1}}-\bfv^{\pi^*})+r^{\pi_{k-1}}_{a'}x^{\pi_k}_{a'}\\
\leq&\bfb^T(\bfv^{\pi_{k-1}}-\bfv^{\pi^*})+r^{\pi_{k-1}}_{a'},
\end{align*}
where the first equality is obtained by rearranging the terms, the second equality follows from Proposition \ref{prop-basic}(v), the first inequality follows from $\tau_{k-1}=\pi^2(\sigma_{k-1})$, namely, $r^{\pi_{k-1}}_a\leq0$ for all $a\in\mcA^2$, the last equality follows from the updating scheme of Algorithm \ref{alg}, namely, $r^{\pi_{k-1}}_a\bfx^{\pi_k}_a=0$ for all $a\in\sigma_k$ and $a\neq a'$, and the last inequality follows from $r^{\pi_{k-1}}_{a'}\leq0$ and Proposition \ref{prop-basic}(iv) for this cycle action $a'$.

Note that $r^{\pi_{k-1}}_{a'}=\Delta_{k-1}$ in Lemma \ref{lemma10}. This together with the above inequality and Lemma \ref{lemma10} yields
\[
\bfb^T(\bfv^{\pi_k}-\bfv^{\pi^*})\leq\bfb^T(\bfv^{\pi_{k-1}}-\bfv^{\pi^*})+r^{\pi_{k-1}}_{a'}\leq\left(1-\frac{1}{n}\right)\bfb^T(\bfv^{\pi_{k-1}}-\bfv^{\pi^*}).
\]
\end{proof}

\subsection{Improvement before creating cycles}
In this subsection, we analyze what happens when no new cycles of Player 1 are created. To simplify the discussion, we make the following definition.
\begin{defi}[{\bf Optimal Cost without New Cycle}]
Let $\Pi_k$ be the set of strategy profiles such that any $\pi=(\sigma,\tau)\in\Pi_k$ satisfies $\tau=\pi^2(\sigma)$ and every Player 1's cycle in $\pi$ is also a cycle in $\pi_k$. Define 
\[
\tilde\nu^*_k=\min_{\pi\in\Pi_k}\bfb^T\bfv^\pi
\]
as the optimal cost for $\pi_k$ without creating any new Player 1's cycle relative to $\pi_k$.
\end{defi}

The following lemma gives a lower bound on the optimal cost without a new Player 1's cycle using the forward progression property of TBDFG.
\begin{lemma}\label{lemma17}
Suppose the game is a TBDFG. Let $\Delta_k=\min\{r^{\pi_k}_a:a\in\mcA^1\}$ be the minimum modified cost of the strategy profile $\pi_k$ at the beginning of the $(k+1)$-th iteration. Then we have
\[
\tilde\nu^*_k\geq\bfb^T\bfv^{\pi_k}-(1-\gamma)n^2|\Delta_k|.
\]
\end{lemma}

\begin{proof}
In the following, we show that for any $\pi\in\Pi_k$ we have
\[
\bfb^T\bfv^\pi\geq\bfb^T\bfv^{\pi_k}-(1-\gamma)n^2|\Delta_k|.
\]
The result then follows by taking the minimum over all $\pi\in\Pi_k$. Observe that
\begin{align*}
\bfb^T\bfv^\pi-\bfb^T\bfv^{\pi_k}\geq&\bfb^T\bfv^{(\sigma,\tau_k)}-\bfb^T\bfv^{\pi_k}\\
=&\left(\bfr^{\pi_k}_{(\sigma,\tau_k)}\right)^T\bfx^{(\sigma,\tau_k)}_+\\
=&\sum_{a\in\sigma\setminus\Cyc(\pi_k)}r^{\pi_k}_ax^{(\sigma,\tau_k)}_a\\
\geq& -\sum_{a\in \sigma\setminus \Cyc(\pi_k)}
x^{(\sigma,\tau_k)}_a|\Delta_k|\\
\geq&-(1-\gamma)n^2|\Delta_k|,
\end{align*}
where the first inequality follows from the optimality of $\tau$ against $\sigma$, the first equality follows from Proposition \ref{prop-basic}(v), and the second equality follows from $r^{\pi_k}_a=0$ for all $a\in\tau_k$ or $a\in\sigma\cap\Cyc(\pi_k)$, and the second inequalty follows from the definition of $\Delta_k$. To justify the last inequality, note that every action in $\sigma\setminus \Cyc(\pi_k)$ must be a path action under $(\sigma,\tau_k)$. Otherwise, such an action would lie on a cycle of $(\sigma,\tau_k)$; by the TBDFG property, this cycle cannot contain Player~2 actions, and hence it would also be a Player~1 cycle in $\pi=(\sigma,\tau)$. Since $\pi\in\Pi_k$, this cycle would be covered by cycles in $\pi_k$, contradicting $a\notin\Cyc(\pi_k)$. The last inequality then follows from Proposition~\ref{prop-basic}(iv), since there are at most $n$ such actions and each path action has flux at most $(1-\gamma)n$.
\end{proof}

We next show that a small improvement is still made when no Player 1's cycles are created for TBDFGs.

\begin{lemma}\label{lemma18}
Suppose the game is a TBDFG. Let $k^+\geq k$ be any iteration before creating a new Player 1's cycle (or algorithm termination). Then for any $k'\in\mathbb{N}\cap(k,k^+]$ the following inequality holds
\[
\bfb^T(\bfv^{\pi_{k'}}-\bfv^{\pi_{k^+}})\leq\left(1-\frac{1}{n^2}\right)\bfb^T(\bfv^{\pi_k}-\bfv^{\pi_{k^+}}).
\]
\end{lemma}

\begin{proof}
If $k^+=k$, then $(k,k^+]$ contains no integer, and the proof is complete.

If $k^+>k$, let $\Delta_k=\min\{r^{\pi_k}_a:a\in\mcA^1\}$ be the modified cost computed at Step 1 of Algorithm \ref{alg} at the $(k+1)$-th iteration, and $a'$ be the action corresponding to this smallest modified cost, which is a path action on $\pi_{k+1}$ by the definition of $k^+$. Then, for any $k'\in(k,k^+]$ we have
\begin{equation}\label{l18-e1}
\bfb^T(\bfv^{\pi_k}-\bfv^{\pi_{k'}})\geq\bfb^T(\bfv^{\pi_k}-\bfv^{\pi_{k+1}})=-\left(\bfr^{\pi_k}_{\pi_{k+1}}\right)^T\bfx^{\pi_{k+1}}_+\geq-r^{\pi_k}_{a'}x^{\pi_{k+1}}_{a'}\geq(1-\gamma)|\Delta_k|,
\end{equation}
where the first inequality follows from Lemma \ref{lemma9}, the first equality follows from Proposition \ref{prop-basic}(v), the second equality follows from the definition of $a'$ and the updating scheme (see Step 1 of Algorithm \ref{alg}), and the last inequality follows from the definition of $\Delta_k$ and Proposition \ref{prop-basic}(iv). Then we have
\begin{align*}
\bfb^T(\bfv^{\pi_{k'}}-\bfv^{\pi_{k^+}})=&\bfb^T(\bfv^{\pi_k}-\bfv^{\pi_{k^+}})-\bfb^T(\bfv^{\pi_k}-\bfv^{\pi_{k'}})\\
\leq&\bfb^T(\bfv^{\pi_k}-\bfv^{\pi_{k^+}})-(1-\gamma)|\Delta_k|\\
\leq&\left(1-\frac{1}{n^2}\right)\bfb^T(\bfv^{\pi_k}-\bfv^{\pi_{k^+}}),
\end{align*}
where the first equality follows by rearranging the terms, the first inequality follows from \eqref{l18-e1}, and the second inequality follows from $\pi_{k^+}\in\Pi_k$ and Lemma \ref{lemma17}.
\end{proof}

Combining the previous two lemmas, we analyze the possible long-run behavior of the algorithm for solving TBDFGs.
\begin{lemma}\label{FG-lemma19}
Suppose the game is a TBDFG. Then after running Algorithm \ref{alg} for $O(n^2\log n)$ iterations from iteration $k$, one of the following must hold:
\begin{enumerate}[label=(\roman*)]
\item the algorithm terminates,
\item a new Player 1's cycle is created,
\item a Player 1's cycle is broken,
\item some action in $\pi_k$ will never appear until a new cycle is created.
\end{enumerate}
\end{lemma}

\begin{proof}
Let $k^+\geq k$ be any iteration before creating a new Player 1's cycle (or algorithm termination or $k^+\geq k$), and let $k'$ be any round between $k$ and $k^+$. By Proposition \ref{prop-basic}(v), we have
\[
(\bfr^{\pi_{k^+}})^T\bfx^{\pi_k}=\bfb^T(\bfv^{\pi_k}-\bfv^{\pi_{k^+}}).
\]

We divide the analysis into two cases. First suppose that there exists an action $a$ used in state $s$ on a path in $\pi_k$ such that $r^{\pi_{k^+}}_ax^{\pi_k}_a\geq(\bfr^{\pi_{k^+}})^T\bfx^{\pi_k}/n$ (note that $(\bfr^{\pi_{k^+}})^T\bfx^{\pi_k}\geq0$). Since $a$ is on a path $x^{\pi_k}_a\leq(1-\gamma)n$, which implies that $(1-\gamma)r^{\pi_{k^+}}_an^2\geq(\bfr^{\pi_{k^+}})^T\bfx^{\pi_k}$. Now if strategy $\pi_{k'}$ uses action $a$, then
\begin{align*}
\bfb^T(\bfv^{\pi_{k'}}-\bfv^{\pi_{k^+}})\geq& (1-\gamma)\left(v^{\pi_{k'}}_s-v^{\pi_{k^+}}_s\right)\\
=&(1-\gamma)\left((c_a+\gamma \bfp_a^T\bfv^{\pi_{k'}})-v^{\pi_{k^+}}_s\right)\\
\geq& (1-\gamma)\left((c_a+\gamma \bfp_a^T\bfv^{\pi_{k^+}})-v^{\pi_{k^+}}_s\right)\\
=&(1-\gamma)r^{\pi_{k^+}}_a\geq(\bfr^{\pi_{k^+}})^T\bfx^{\pi_k}/n^2,
\end{align*}
where the first equality follows from Proposition \ref{prop-basic}(i), the first and second inequalities follow from Lemma \ref{lemma9}, the second equality follows from \eqref{opt}, and the last inequality follows from $(1-\gamma)r^{\pi_{k^+}}_an^2\geq(\bfr^{\pi_{k^+}})^T\bfx^{\pi_k}$.

In the second case, there is no action $a$ on a path in $\pi_k$ satisfying $r^{\pi_{k^+}}_ax^{\pi_k}_a\geq(\bfr^{\pi_{k^+}})^T\bfx^{\pi_k}/n$. The remaining portion of $(\bfr^{\pi_{k^+}})^T\bfx^{\pi_k}$ is due to cycles, so there must be some cycle $C$ consisting of actions $\{a_1,\dots,a_t\}$ and corresponding states $\{s_1,\dots,s_t\}$ such that $\sum_{a\in C}r^{\pi_{k^+}}_ax^{\pi_k}_a\geq(\bfr^{\pi_{k^+}})^T\bfx^{\pi_k}/n$. Note that $(\bfr^{\pi_{k^+}})^T\bfx^{\pi_k}\geq0$ and it follows from the optimality of the Player 2's strategy in $k^+$ that $r^{\pi_{k^+}}_a\leq0$ for all $a\in\mcA^2$. These together implies that $C$ can only be Player 1's cycle. 

All flux in $C$ first enters $C$ either from a path ending at $C$ or from the initial unit of flux placed on some state $s$ in $C$. If $y_s\geq1-\gamma$ units of (scaled) flux first enter $C$ at state $s$ in strategy $\pi_k$, then that flux incurs $y_s(v^{\pi_k}_s-v^{\pi_{k^+}}_s)$ cost with respect to $\bfr^{\pi_{k^+}}$, so $\sum_{a\in C}r^{\pi_{k^+}}_ax^{\pi_k}_a=\sum_{s\in C}y_s(v^{\pi_k}_s-v^{\pi_{k^+}}_s)$. Moreover, each term $v^{\pi_k}_s-v^{\pi_{k^+}}_s$ is nonnegative (see Lemma \ref{lemma9}). Under the scaled flux convention $\bfb=(1-\gamma)\bfe$, each of the $n$ starting states contributes at most $(1-\gamma)$ units of first-entry mass to any fixed state of the cycle, because the deterministic trajectory can first enter the cycle only once, we have $y_s\leq n(1-\gamma)$. Also note that $\sum_{s\in C}(v^{\pi_k}_s-v^{\pi_{k^+}}_s)=\sum_{a\in C}r^{\pi_{k^+}}_a/(1-\gamma)$. Therefore, $n\sum_{a\in C}r^{\pi_{k^+}}_a\geq\sum_{a\in C}r^{\pi_{k^+}}_ax^{\pi_k}_a$ implying $n^2\sum_{a\in C}r^{\pi_{k^+}}_a\geq(\bfr^{\pi_{k^+}})^T\bfx^{\pi_k}$.

As long as cycle $C$ is intact, each $a\in C$ has $1$ (scaled) flux from states in $C$ (see Proposition \ref{prop-basic}(iv)), so if $C$ is in strategy $\pi_{k'}$ then
\begin{align*}
\bfb^T(\bfv^{\pi_{k'}}-\bfv^{\pi_{k^+}})\geq(1-\gamma)\sum_{s\in C}\left(v^{\pi_{k'}}_s-v^{\pi_{k^+}}_s\right)=\sum_{a\in C}r^{\pi_{k^+}}_a\geq(\bfr^{\pi_{k^+}})^T\bfx^{\pi_k}/n^2.
\end{align*}

Now if $\log_{n^2/(n^2-1)}n^2$ iterations occur between $\pi_k$ and $\pi_{k'}$, Lemma \ref{lemma18} implies that
\[
\bfb^T(\bfv^{\pi_{k'}}-\bfv^{\pi_{k^+}})<\bfb^T(\bfv^{\pi_k}-\bfv^{\pi_{k^+}})/n^2=(\bfr^{\pi_{k^+}})^T\bfx^{\pi_k}/n^2,
\]
where the equality follows from Proposition \ref{prop-basic}(v). In the first case action $a$ cannot appear in $\pi_{k'}$, and in the second case Player 1's cycle $C$ must be broken in $\pi_{k'}$. This takes $\log_{n^2/(n^2-1)}n^2=O(n^2\log n)$ iterations if no new cycles interrupt the process.
\end{proof}

\subsection{Strong polynomial complexity}
In this subsection, we prove our main result that Algorithm \ref{alg} is strongly polynomial for TBDFGs. Using the previous lemma, we show that the algorithm cannot continue making only slight improvements indefinitely without either creating a new Player 1' cycle or terminating.

\begin{lemma}\label{FG-cycle}
Suppose the game is a TBDFG. Either the algorithm terminates or a new Player 1's cycle is created after $O(n^2m\log n)$ iterations of Algorithm \ref{alg}.
\end{lemma}

\begin{proof}
Let $\pi_{k_0}$ be a strategy after a new Player 1's cycle is created, and consider the policies $\pi_{k_1},\pi_{k_2},\dots$, each separated by $O(n^2\log n)$ iterations. If no Player 1's cycle is created, then by Lemma \ref{FG-lemma19}, each of these policies $\pi_{k_i}$ has either broken another Player 1's cycle in $\pi_{k_0}$ or contains an action that cannot appear in $\pi_{k_j}$ for all $j>i$. There are at most $n$ Player 1's cycles in $\pi_{k_0}$ and at most $m$ actions that can be eliminated. Therefore, after $(m+n)O(n^2\log n)$ iterations, the algorithm must terminate or create a new cycle.
\end{proof}

Using the significant improvement achieved when a new cycle is created with Player 1's pivoted in action, we establish the following lemma.

\begin{lemma}\label{FG-elim}
Suppose the game is a TBDFG. Let $\pi$ be a strategy profile at some iteration of Algorithm \ref{alg}. Starting from $\pi$, after $O(n\log n)$ Player 1's cycle-creating rounds, some action in $\pi\cap\mcA^1$ is either eliminated from cycles for the remainder of the algorithm or entirely eliminated from policies for the remainder of the algorithm.
\end{lemma}

\begin{proof}
Consider a strategy profile $\pi=(\sigma,\tau)$ where $\tau=\pi^2(\sigma)$ with respect to the optimal gains $\bfr^{\pi^*} $ where $\pi^*=(\pi_*^1,\pi_*^2)$ (see Proposition \ref{prop-basic}(vi)).

There is an action such that $r^{\pi^*}_ax^{\pi}_a\geq(\bfr^{\pi^*})^T\bfx^\pi/n=\bfb^T(\bfv^\pi-\bfv^{\pi^*})/n\geq0$. Notice that $\bfr^{\pi^*}$ is the gain vector for the optimal strategy, $r^{\pi^*}_{a'}\geq0$ for all $a'\in\mcA^1$ and $r^{\pi^*}_{a'}\leq0$ for all $a'\in\mcA^2$, thus we have $a\in\mcA^1$. If $a$ is on a path in $\pi$, then $1-\gamma\leq x^\pi_a\leq(1-\gamma)n$, so $(1-\gamma)r^{\pi^*}_a\geq(\bfr^{\pi^*})^T\bfx^\pi/n^2$, and if $a$ is on a cycle in $\pi$, then $1\leq x^\pi_a\leq n$, so $r^{\pi^*}_a\geq(\bfr^{\pi^*})^T\bfx^\pi/n^2$.

If $\pi'=(\sigma',\tau')$ where $\tau'=\pi^2(\sigma')$ is any strategy containing $a\in\mcA^1$ and $a$ is on a path in $\pi$, then we have
\begin{align*}
&\bfb^T(\bfv^{\pi'}-\bfv^{\pi^*})\geq\bfb^T\left(\bfv^{(\sigma',\pi^2_*)}-\bfv^{\pi^*}\right)\geq(\bfr^{\pi^*})^T\bfx^{(\sigma',\pi^2_*)}\\
&=\sum_{a'\in\sigma'}r^{\pi^*}_{a'}x^{(\sigma',\pi^2_*)}_{a'}\geq r^{\pi^*}_ax^{(\sigma',\pi^2_*)}_a\geq (1-\gamma)r^{\pi^*}_a\geq(\bfr^{\pi^*})^T\bfx^\pi/n^2,
\end{align*}
where the first inequality follows from the optimality of $\tau'$, the second inequality follows from Proposition \ref{prop-basic}(v), the first equality and the third inequality follow from Lemmas \ref{l-opt} and Proposition \ref{prop-basic}(iv), the second last inequality follows from $x^\pi_a\geq(1-\gamma)$ and the forward progression property of TBDFG, and the last inequality follows from $(1-\gamma)r^{\pi^*}_a\geq(\bfr^{\pi^*})^T\bfx^\pi/n^2$. Similarly, if $\pi'$ contains $a\in\mcA^1$ on a cycle and $a$ is on a cycle in $\pi$, we have $\bfb^T(\bfv^{\pi'}-\bfv^{\pi^*})\geq(\bfr^{\pi^*})^T\bfx^\pi/n^2$

Now by Lemmas \ref{lemma9} and \ref{lem:sig}, and there are more than $\log_{n/(n-1)}n^2=O(n\log n)$ new Player 1's cycles created between policies $\pi$ and $\pi'$ then
\[
\bfb^T(\bfv^{\pi'}-\bfv^{\pi^*})<\left(1-\frac{1}{n}\right)^{\log_{n/(n-1)}n^2}\bfb^T(\bfv^\pi-\bfv^{\pi^*})\leq(\bfr^{\pi^*})^T\bfx^\pi/n^2,
\]
where the first inequality follows from Lemmas \ref{lemma9} and \ref{lem:sig}, and the equality follows from Proposition \ref{prop-basic}(v). Therefore, if $\pi$ contained $a$ on a path, then $a$ cannot appear in any strategy after $\pi'$ for the remainder of the algorithm. If $\pi$ contained $a$ on a cycle, then $a$ cannot appear in a cycle after $\pi'$, although it may still appear on a path later in the algorithm.
%
%
%
\end{proof}

We close the section with the complexity bound for Algorithm \ref{alg} and show that TBDFG is strongly polynomial-time solvable.

\begin{thm}\label{thm}
The simple strategy-iteration method in Algorithm \ref{alg} converges in at most $O(n^6m^4\log^4 n)$ pivot steps of simplex method on TBDFGs with uniform discount $\gamma\in(0,1)$. Therefore, the simple strategy-iteration method is strongly polynomial for the TBDFG.

\end{thm}
\begin{proof}
Consider the policies $\pi_{k_0},\pi_{k_1},\dots$ such that $O(n\log n)$ new Player 1's cycles are created between $\pi_{k_i}$ and $\pi_{k_{i+1}}$. By Lemma \ref{FG-elim}, each $\pi_{k_i}$ contains an action that is either eliminated entirely from $\pi_{k_j}$ for $j>i$ or eliminated from cycles. Each action can be eliminated from both cycles and paths, so after $2m$ such rounds of $O(n\log n)$ new Player 1's cycles, the algorithm has converged. By Proposition \ref{thm-MDP} and \ref{FG-cycle}, each outer iteration takes $O(n^3m^2\log^2 n)$ pivot steps of simplex method, and Player 1's cycles are created every $O(n^2m\log n)$ iterations. Thus the algorithm terminates in $O(n^6m^4\log^4 n)$ total pivot steps of simplex method. 


\end{proof}

\section{A structure-aware SCC algorithm for TBDFGs}\label{sec:specialized}
The previous section established the main result for Algorithm \ref{alg}: although the strategy-iteration method applies to general TBSGs, it follows a strongly polynomial trajectory whenever the instance is a TBDFG, without requiring prior knowledge of this structure. We now turn to the complementary, structure-aware setting in which the TBDFG property is known or certified in advance and can be exploited explicitly. The absence of directed cycles involving states controlled by both players allows the game graph to be decomposed into one-player recurrent components arranged in an acyclic dependency order. Although this approach is less general than simple strategy iteration, the resulting decomposition yields a specialized graph-based algorithm with improved complexity. Thus, structural knowledge is unnecessary for strong polynomiality, but it enables a more efficient algorithm.

For clarity of our discussion, we introduce the following graph-theoretic terminology.

\begin{defi}[{\bf Strongly connected directed graph}]\label{def:scc}
A graph is called strongly connected if there is a path in each direction between each pair of vertices of the graph. Additionally, a strongly connected component (SCC) of a directed graph is a maximal strongly connected subset of vertices.
\end{defi}

\begin{defi}[{\bf Condensation graph}]
\label{def:cg}
Let $\mcG=(\mcS,\mcA)$ be the directed graph of a TBDFG, and let $\mcC_1,\ldots,\mcC_K$ be its strongly connected components. The \emph{condensation graph} of $G$ is the directed graph $\mcG^{\rm SCC}=(\mcN,\mcE)$, where each node in $\mcN=\{\mcC_1,\ldots,\mcC_K\}$ represents one strongly connected component of $G$. There is a directed edge $(\mcC_i,\mcC_j)\in \mcE$ with $i\neq j$ if and only if there exist states $s\in\mcC_i$ and $s'\in\mcC_j$ such that $(s,s')=a\in\mcA$.
\end{defi}

The key observation is that the condensation graph of a TBDFG is acyclic, and each of its nodes corresponds to a strongly connected component controlled by a single player. Thus, at a high level, a specialized algorithm can process these components in reverse topological order. When a component is processed, the values of all successor components have already been determined. The component therefore reduces to a deterministic MDP with fixed boundary values on its outgoing actions. Since deterministic MDPs admit strongly polynomial algorithms, this componentwise decomposition provides the structural foundation for a strongly polynomial algorithm tailored to TBDFGs.

We first prove our observed graph-theoretic facts that justify this backward propagation procedure in the following lemmas. They are direct consequences of the forward-progression property of TBDFGs.

\begin{lemma}\label{lem:one_cmpt_one_player}
Every strongly connected component of the graph of a TBDFG contains states controlled by at most one player.
\end{lemma}

\begin{proof}
Suppose for contradiction that there exists a strongly connected component $\mcC$ containing states controlled by both players. Then there exist states$s_1 \in \mcS^1$ and $s_2 \in \mcS^2$ such that $s_1,s_2\in \mcC$.

Since $\mcC$ is strongly connected, there is a directed path from $s_1$ to $s_2$. Along this path, the controller of the states must change at least once. Hence, there exist two consecutive states $u$ and $v$ on the path such that $u$ and $v$ are controlled by different players. In particular, the directed edge from $u$ to $v$ belongs to $\mcC$. Again using the strong connectivity of $\mcC$, there exists a directed path from $v$ back to $u$. Deleting repeated states if necessary, we may assume that this path is without any duplicate states. Concatenating the edge $(u,v)$ with this path yields a directed cycle that contains states controlled by both players, which contradicts the defining property of a TBDFG in Definition~\ref{tbdfg}.

Therefore, no strongly connected component can contain states controlled by both players, and the conclusion follows.
\end{proof}

\begin{lemma}\label{lem:one_side_path}
Let $\mcC_1$ and $\mcC_2$ be two strongly connected components whose states are controlled by Players~1 and~2, respectively. Then paths from $\mcC_1$ to $\mcC_2$ and from $\mcC_2$ to $\mcC_1$ cannot both exist.
\end{lemma}
\begin{proof}Suppose for contradiction that there exist states $s_1,s_1'\in \mcC_1$ and $s_2,s_2'\in \mcC_2$ such that there is a path from $s_1$ to $s_2$ and a path from $s_2'$ to $s_1'$. Since $\mcC_1$ is strongly connected, there is a path from $s_1'$ to $s_1$. Similarly, since $\mcC_2$ is strongly connected, there is a path from $s_2$ to $s_2'$. Concatenating these paths gives a path from $s_1$ to $s_2$ and a path from $s_2$ back to $s_1$. Therefore, $s_1\in\mcC_2$ by Definition \ref{def:scc} and this contradicts Lemma \ref{lem:one_cmpt_one_player}. Hence the conclusion holds.
\end{proof}

\begin{lemma}\label{acyclic}
Let $\mcG^{\mathrm{SCC}}$ be the condensation graph of a TBDFG. Then $\mcG^{\mathrm{SCC}}$ is a directed acyclic graph.
\end{lemma}
\begin{proof}
This is implied directly from Lemma \ref{lem:one_side_path}. If there is a cycle in the condensation graph, then there must be two paths connecting two strongly connected components from both directions, conflicting the results in Lemma \ref{lem:one_side_path}.
\end{proof}

Lemma~\ref{acyclic} together with Lemma~\ref{lem:one_cmpt_one_player}, identifies the structural property that makes TBDFGs algorithmically simpler than general TBSGs. The game can be decomposed into strongly connected components that are ordered acyclically, and each such component consists of states controlled by a single player. 
\begin{figure}[H]
\centering

\begin{tikzpicture}[
scale=0.78,
state1/.style={circle, draw=blue!60, fill=blue!20, thick, minimum size=4mm, inner sep=0pt, font=\tiny},
state2/.style={circle, draw=red!60, fill=red!20, thick, minimum size=4mm, inner sep=0pt, font=\tiny},
legendstate1/.style={circle, draw=blue!60, fill=blue!20, thick, minimum size=3.5mm, inner sep=0pt},
legendstate2/.style={circle, draw=red!60, fill=red!20, thick, minimum size=3.5mm, inner sep=0pt},
scc1/.style={draw=blue!35, fill=blue!3, line width=0.45pt, rounded corners=6pt,
              minimum width=1.45cm, minimum height=1.20cm, inner sep=0pt},
scc2/.style={draw=red!35, fill=red!3, line width=0.45pt, rounded corners=6pt,
              minimum width=1.45cm, minimum height=1.20cm, inner sep=0pt},
blackedge/.style={->, line width=0.45pt},
caption/.style={font=\scriptsize},
>=stealth,
font=\small
]


\begin{scope}[on background layer]
    \node[scc1] (Abox) at (0.00,0.25) {};
    \node[scc2] (Bbox) at (2.65,0.25) {};
    \node[scc1] (Cbox) at (5.30,0.25) {};
    \node[scc2] (Dbox) at (7.95,0.25) {};

    \node[scc2] (Ebox) at (0.00,-2.55) {};
    \node[scc1] (Fbox) at (2.65,-2.55) {};
    \node[scc2] (Gbox) at (5.30,-2.55) {};
    \node[scc1] (Hbox) at (7.95,-2.55) {};
\end{scope}


\node[legendstate1] at (12.00,0.65) {};
\node[anchor=west, font=\scriptsize] at (12.30,0.65) {Player 1 state $(\mcS^1)$};

\node[legendstate2] at (12.00,0.15) {};
\node[anchor=west, font=\scriptsize] at (12.30,0.15) {Player 2 state $(\mcS^2)$};

\draw[blackedge] (11.80,-0.35) -- (12.20,-0.35);
\node[anchor=west, font=\scriptsize] at (12.30,-0.35) {Action};

\node[scc1, minimum width=0.55cm, minimum height=0.32cm] at (12.00,-0.90) {};
\node[anchor=west, font=\scriptsize] at (12.30,-0.90) {Player 1 SCC};

\node[scc2, minimum width=0.55cm, minimum height=0.32cm] at (12.00,-1.40) {};
\node[anchor=west, font=\scriptsize] at (12.30,-1.40) {Player 2 SCC};


\node[state1] (a1) at (-0.35,0.35) {1};
\node[state1] (a2) at (0.25,0.65) {2};
\node[state1] (a3) at (0.50,-0.05) {3};

\draw[dashed,blue,->,thick] (a1) -- (a2);
\draw[dashed,blue,->,thick] (a2) -- (a3);
\draw[dashed,blue,->,thick] (a3) -- (a1);

\node[state2] (b1) at (2.35,0.35) {4};
\node[state2] (b2) at (2.95,-0.20) {5};

\draw[dashed,red,->,thick,bend left=25] (b1) to (b2);
\draw[dashed,red,->,thick,bend left=25] (b2) to (b1);

\node[state1] (c1) at (5.30,0.15) {6};

\draw[dashed,blue,->,thick]
    (c1) to[out=60,in=120,looseness=8] (c1);

\node[state2] (d1) at (7.60,0.55) {7};
\node[state2] (d2) at (8.30,0.55) {8};
\node[state2] (d3) at (7.95,-0.15) {9};

\draw[dashed,red,->,thick] (d1) -- (d2);
\draw[dashed,red,->,thick] (d2) -- (d3);
\draw[dashed,red,->,thick] (d3) -- (d1);


\node[state2] (e1) at (-0.35,-2.20) {10};
\node[state2] (e2) at (0.35,-2.20) {11};
\node[state2] (e3) at (0.35,-2.90) {12};
\node[state2] (e4) at (-0.35,-2.90) {13};

\draw[dashed,red,->,thick] (e1) -- (e2);
\draw[dashed,red,->,thick] (e2) -- (e3);
\draw[dashed,red,->,thick] (e3) -- (e4);
\draw[dashed,red,->,thick] (e4) -- (e1);

\node[state1] (f1) at (2.35,-2.25) {14};
\node[state1] (f2) at (2.95,-2.80) {15};

\draw[dashed,blue,->,thick,bend left=25] (f1) to (f2);
\draw[dashed,blue,->,thick,bend left=25] (f2) to (f1);

\node[state2] (g1) at (5.00,-2.25) {16};
\node[state2] (g2) at (5.60,-2.80) {17};

\draw[dashed,red,->,thick,bend left=25] (g1) to (g2);
\draw[dashed,red,->,thick,bend left=25] (g2) to (g1);

\node[state1] (h1) at (7.95,-2.55) {18};

\draw[dashed,blue,->,thick]
    (h1) to[out=300,in=20,looseness=8] (h1);


\draw[blackedge] (a3) -- (b1);
\draw[blackedge] (a2) -- (b1);
\draw[blackedge] (b1) -- (c1);
\draw[blackedge] (c1) -- (d1);

\draw[blackedge] (e2) -- (f1);
\draw[blackedge] (f2) -- (g1);
\draw[blackedge] (g2) -- (h1);

\draw[blackedge] (a1) -- (e1);
\draw[blackedge] (b2) -- (f1);
\draw[blackedge] (c1) -- (g1);
\draw[blackedge] (d3) -- (h1);

\draw[blackedge, bend left=12] (a3) to (f1);
\draw[blackedge, bend left=12] (b2) to (g1);
\draw[blackedge, bend left=12] (c1) to (h1);
\draw[blackedge, bend right=15] (d3) to (g2);


\node[caption] at (0.00,1.32) {$\mcC_1$};
\node[caption] at (2.65,1.32) {$\mcC_2$};
\node[caption] at (5.30,1.32) {$\mcC_3$};
\node[caption] at (7.95,1.32) {$\mcC_4$};

\node[caption] at (0.00,-3.58) {$\mcC_5$};
\node[caption] at (2.65,-3.58) {$\mcC_6$};
\node[caption] at (5.30,-3.58) {$\mcC_7$};
\node[caption] at (7.95,-3.58) {$\mcC_8$};

\path[use as bounding box] (-1.25,-4.05) rectangle (15.20,1.75);

\end{tikzpicture}

\caption{Strongly connected components in a two-player game graph. }
\label{fig:scc_twoplayer}

\end{figure}
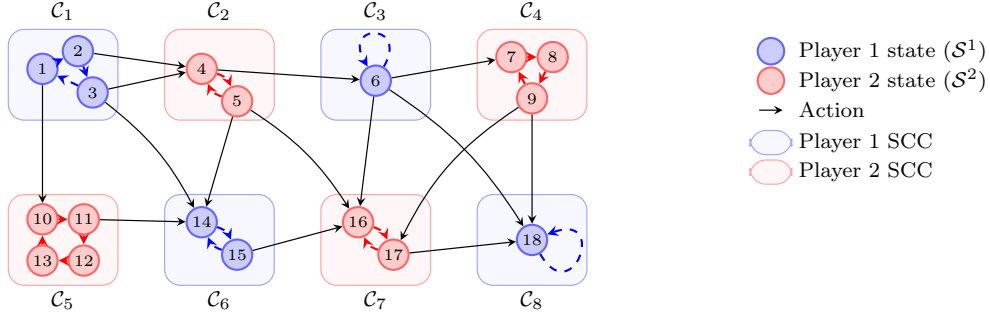

Figure~\ref{fig:scc_twoplayer} illustrates this decomposition. The shaded regions represent strongly connected components; within each region, recurrence is confined to one player's states, while the black edges represent actions between different components. After contracting each shaded component to a single node, these inter-component edges form the condensation graph, which has no directed cycles. Therefore, once the values of all successor components have been computed, the remaining subproblem inside a component is a one-player deterministic MDP with fixed boundary values on actions leaving the component. This observation motivates the following backward propagation algorithm.

\begin{algorithm}[H]
\caption{Backward SCC propagation algorithm for TBDFGs}
\label{alg:scc_backward}
\begin{algorithmic}[1]
\REQUIRE A TBDFG with directed graph $\mcG=(\mcS,\mcA,\bfc,\bfP,\gamma)$. Set $k=1$.

\STATE Compute the strongly connected components $\mcC_1,\ldots,\mcC_K$ of $\mcG$ and construct the condensation graph $\mcG^{\rm SCC}$. Compute a reverse topological ordering of $\mcG^{\rm SCC}$, relabeled so that every edge $(\mcC_i, \mcC_j)$ satisfies $i>j$.
\STATE Consider the component $\mcC_k$, treat the values of all successor components of $\mcC_k$ as fixed boundary values, and solve the resulting one-player deterministic MDP restricted to $\mcC_k$ using simplex methods with simple pivoting rule. Store the values $v_s$ and the selected optimal actions $\pi_s$ for all $s\in\mcC_k$.
\STATE Terminate the algorithm if $k=K$. Otherwise, set $k\leftarrow k+1$ and go to Step 2.
\end{algorithmic}
\end{algorithm}

\begin{rem}
\begin{enumerate}[label=(\roman*)]

\item The strongly connected components of a directed graph can be computed by standard graph-search algorithms, such as Tarjan's algorithm~\cite{tarjan1972depth}. This algorithm also allows us to construct the condensation graph efficiently. It compute all SCCs in time $O(n+m)$; see \cite{tarjan1972depth}. After the SCCs are identified, the condensation graph can be constructed by scanning each action $a\in\mcA$ once and adding an edge between the corresponding components. This additional step also takes $O(n+m)$ time. Hence the total running time of Step 1 is $O(n+m)$.
\item The correctness of Algorithm~\ref{alg:scc_backward} follows from the acyclic structure of the condensation graph; see Lemma \ref{acyclic}. When a component is processed in reverse topological order, every action leaving the component enters a successor component whose value has already been computed. Hence these outgoing actions can be treated as boundary actions with fixed continuation values. Since each strongly connected component contains states controlled by only one player, the remaining local problem is a deterministic MDP. Solving these MDP subproblems backward over the condensation graph therefore yields the value vector and an optimal stationary strategy for the entire TBDFG.
\end{enumerate}
\end{rem}

We next prove that Algorithm \ref{alg:scc_backward} computes optimal strategy profile for the TBDFG in strongly polynomial time as stated in the following theorem.
\begin{thm}
    \label{thm:scc_backward_correct}
    Algorithm~\ref{alg:scc_backward} correctly computes the value vector and an optimal strategy profile for a TBDFG. Moreover, it requires at most $O(n^3m^2\log^2 n)$ pivot steps of simplex method.
\end{thm}

\begin{proof}
Note that $\mcC_1,\ldots,\mcC_K$ is reversely ordered topologically so that every edge in the condensation graph points from a larger index to a smaller index. We prove the result by induction. For $\mcC_1$, it follows from Lemma~\ref{lem:one_cmpt_one_player} that the component contains states controlled by only one player. Therefore, the restriction of the game to it is a deterministic MDP, and solving it gives the correct values and optimal actions on the component.

Now suppose for induction that all components $\mcC_j$ with $j<k$ have already been solved correctly. Consider $\mcC_k$. Every action leaving $\mcC_k$ enters some successor component $\mcC_j$ with $j<k$, whose values are already known by the induction hypothesis. Thus the continuation values of all outgoing actions from $\mcC_k$ are fixed. After incorporating these fixed boundary values, the remaining problem on $\mcC_k$ is again a deterministic MDP, because all states in $\mcC_k$ are controlled by a single player. Solving this local MDP gives the correct values and optimal actions on $\mcC_k$.

By induction, the algorithm computes the correct values on all strongly connected components. The local optimal actions selected across all components form an optimal strategy profile for the original TBDFG.

It remains to bound the number of simplex pivot steps. Let $n_k$ and $m_k$ denote the number of states and actions associated with component $\mcC_k$, respectively. By Proposition~\ref{thm-MDP}, solving the deterministic MDP on $\mcC_k$ requires at most $O(n_k^3 m_k^2 \log^2 n_k)$ pivot steps of the simplex method. Since $\sum_{k=1}^K n_k \le n$, $\sum_{k=1}^K m_k \le m$, and $\log n_k \le \log n$, the total number of pivot steps is bounded by $ O(n^3m^2\log^2 n)$. The additional graph operations, including computing the strongly connected components and a topological ordering of the condensation graph, can be performed in linear time in the size of the graph and do not affect the stated pivot-step bound. This completes the proof.
\end{proof}

\section{Summary of contributions}\label{sec:summary}

This paper studies TBDFGs, a class of deterministic turn-based games in which every directed cycle is controlled by a single player. We show that this forward structure yields strong algorithmic consequences. First, the general strategy-iteration method in Algorithm~\ref{alg} follows a discount-independent strongly polynomial complexity on every TBDFG, even though the algorithm neither recognizes nor explicitly exploits the forward property. This is a structure-oblivious guarantee for a standard strategy-iteration method. Second, when the TBDFG structure is known in advance, the game can be decomposed into single-player strongly connected components and solved by a specialized graph-based algorithm with an improved pivot bound. Together, these results show that forward structure is favorable for algorithmic convergence in two complementary ways: it guarantees efficient convergence of a general-purpose strategy-iteration method and, when explicitly identified, supports a faster structure-aware algorithm.


\bibliographystyle{abbrv}
\bibliography{ref}

@article{shapley1953stochastic,
  title={Stochastic games},
  author={Shapley, Lloyd S},
  journal={Proceedings of the national academy of sciences},
  volume={39},
  number={10},
  pages={1095--1100},
  year={1953},
  publisher={National Academy of Sciences}
}

@article{hansen2013strategy,
  title={Strategy iteration is strongly polynomial for 2-player turn-based stochastic games with a constant discount factor},
  author={Hansen, Thomas Dueholm and Miltersen, Peter Bro and Zwick, Uri},
  journal={Journal of the ACM (JACM)},
  volume={60},
  number={1},
  pages={1--16},
  year={2013},
  publisher={ACM New York, NY, USA}
}

@article{post2015simplex,
  title={The simplex method is strongly polynomial for deterministic Markov decision processes},
  author={Post, Ian and Ye, Yinyu},
  journal={Mathematics of Operations Research},
  volume={40},
  number={4},
  pages={859--868},
  year={2015},
  publisher={INFORMS}
}

@article{ye2011simplex,
  author  = {Ye, Yinyu},
  title   = {The Simplex and Policy-Iteration Methods Are Strongly Polynomial for the Markov Decision Problem with a Fixed Discount Rate},
  journal = {Mathematics of Operations Research},
  volume  = {36},
  number  = {4},
  pages   = {593--603},
  year    = {2011},
  doi     = {10.1287/moor.1110.0516}
}

@book{puterman1994markov,
  title={Markov decision processes: discrete stochastic dynamic programming},
  author={Puterman, Martin L},
  year={2014},
  publisher={John Wiley \& Sons}
}

@book{howard1960dynamic,
  title={Dynamic programming and markov processes.},
  author={Howard, Ronald A},
  year={1960},
  publisher={John Wiley}
}

@book{filar1997competitive,
  title={Competitive Markov decision processes},
  author={Filar, Jerzy and Vrieze, Koos},
  year={2012},
  publisher={Springer Science \& Business Media}
}

@article{gillette1957stochastic,
  title={Stochastic games with zero stop probabilities},
  author={Gillette, Dean},
  journal={Contributions to the Theory of Games III},
  volume={39},
  pages={179--187},
  year={1957}
}

@article{condon1992complexity,
  title={The complexity of stochastic games},
  author={Condon, Anne},
  journal={Information and Computation},
  volume={96},
  number={2},
  pages={203--224},
  year={1992},
  publisher={Elsevier}
}

@article{ludwig1995subexponential,
  title={A subexponential randomized algorithm for the simple stochastic game problem},
  author={Ludwig, Walter},
  journal={Information and computation},
  volume={117},
  number={1},
  pages={151--155},
  year={1995},
  publisher={Elsevier}
}

@article{ehrenfeucht1979positional,
  title={Positional strategies for mean payoff games},
  author={Ehrenfeucht, Andrzej and Mycielski, Jan},
  journal={International Journal of Game Theory},
  volume={8},
  number={2},
  pages={109--113},
  year={1979},
  publisher={Springer}
}

@article{zwick1996complexity,
  title={The complexity of mean payoff games on graphs},
  author={Zwick, Uri and Paterson, Mike},
  journal={Theoretical Computer Science},
  volume={158},
  number={1-2},
  pages={343--359},
  year={1996},
  publisher={Elsevier}
}

@inproceedings{emerson1991tree,
  author    = {E. Allen Emerson and Charanjit S. Jutla},
  title     = {Tree Automata, {$\mu$}-Calculus and Determinacy},
  booktitle = {Proceedings of the 32nd Annual Symposium on Foundations of Computer Science},
  pages     = {368--377},
  publisher = {IEEE Computer Society},
  year      = {1991},
  doi       = {10.1109/SFCS.1991.185392}
}

@inproceedings{calude2017deciding,
  title={Deciding parity games in quasipolynomial time},
  author={Calude, Cristian S and Jain, Sanjay and Khoussainov, Bakhadyr and Li, Wei and Stephan, Frank},
  booktitle={Proceedings of the 49th Annual ACM SIGACT Symposium on Theory of Computing},
  pages={252--263},
  year={2017}
}

@article{papadimitriou1987complexity,
  title={The complexity of Markov decision processes},
  author={Papadimitriou, Christos H and Tsitsiklis, John N},
  journal={Mathematics of operations research},
  volume={12},
  number={3},
  pages={441--450},
  year={1987},
  publisher={INFORMS}
}

@inproceedings{littman1995complexity,
  author    = {Michael L. Littman and Thomas L. Dean and Leslie Pack Kaelbling},
  title     = {On the Complexity of Solving Markov Decision Problems},
  booktitle = {Proceedings of the Eleventh Conference on Uncertainty in Artificial Intelligence},
  pages     = {394--402},
  publisher = {Morgan Kaufmann},
  year      = {1995},
  doi       = {10.5555/2074158.2074203}
}

@inproceedings{mansour1999complexity,
  author    = {Yishay Mansour and Satinder Singh},
  title     = {On the Complexity of Policy Iteration},
  booktitle = {Proceedings of the Fifteenth Conference on Uncertainty in Artificial Intelligence},
  pages     = {401--408},
  publisher = {Morgan Kaufmann},
  year      = {1999},
  doi       = {10.5555/2073796.2073842}
}

@article{melekopoglou1994complexity,
  title={On the complexity of the policy improvement algorithm for Markov decision processes},
  author={Melekopoglou, Mary and Condon, Anne},
  journal={ORSA Journal on Computing},
  volume={6},
  number={2},
  pages={188--192},
  year={1994},
  publisher={INFORMS}
}

@inproceedings{fearnley2010exponential,
  title={Exponential lower bounds for policy iteration},
  author={Fearnley, John},
  booktitle={International Colloquium on Automata, Languages, and Programming},
  pages={551--562},
  year={2010},
  organization={Springer}
}

@inproceedings{hollanders2012complexity,
  title={The complexity of policy iteration is exponential for discounted Markov decision processes},
  author={Hollanders, Romain and Delvenne, Jean-Charles and Jungers, Rapha{\"e}l M},
  booktitle={2012 IEEE 51st IEEE Conference on Decision and Control (CDC)},
  pages={5997--6002},
  year={2012},
  organization={IEEE}
}

@inproceedings{andersson2009complexity,
  title={The complexity of solving stochastic games on graphs},
  author={Andersson, Daniel and Miltersen, Peter Bro},
  booktitle={International Symposium on Algorithms and Computation},
  pages={112--121},
  year={2009},
  organization={Springer}
}

@inproceedings{hansen2011exact,
  title={Exact algorithms for solving stochastic games},
  author={Hansen, Kristoffer Arnsfelt and Koucky, Michal and Lauritzen, Niels and Miltersen, Peter Bro and Tsigaridas, Elias P},
  booktitle={Proceedings of the forty-third annual ACM symposium on Theory of computing},
  pages={205--214},
  year={2011}
}

@article{friedmann2011exponential,
  title={An exponential lower bound for the latest deterministic strategy iteration algorithms},
  author={Friedmann, Oliver},
  journal={Logical Methods in Computer Science},
  volume={7},
  year={2011},
  publisher={Episciences. org}
}

@article{silver2016mastering,
  title={Mastering the game of Go with deep neural networks and tree search},
  author={Silver, David and Huang, Aja and Maddison, Chris J and Guez, Arthur and Sifre, Laurent and Van Den Driessche, George and Schrittwieser, Julian and Antonoglou, Ioannis and Panneershelvam, Veda and Lanctot, Marc and others},
  journal={nature},
  volume={529},
  number={7587},
  pages={484--489},
  year={2016},
  publisher={Nature Publishing Group UK London}
}

@article{silver2017mastering,
  title={Mastering the game of go without human knowledge},
  author={Silver, David and Schrittwieser, Julian and Simonyan, Karen and Antonoglou, Ioannis and Huang, Aja and Guez, Arthur and Hubert, Thomas and Baker, Lucas and Lai, Matthew and Bolton, Adrian and others},
  journal={nature},
  volume={550},
  number={7676},
  pages={354--359},
  year={2017},
  publisher={Nature Publishing Group UK London}
}

@article{jia2020towards,
  title={Towards solving 2-{TBSG} efficiently},
  author={Jia, Zeyu and Wen, Zaiwen and Ye, Yinyu},
  journal={Optimization Methods and Software},
  volume={35},
  number={4},
  pages={706--721},
  year={2020},
  publisher={Taylor \& Francis}
}

@article{tarjan1972depth,
  title={Depth-first search and linear graph algorithms},
  author={Tarjan, Robert},
  journal={SIAM journal on computing},
  volume={1},
  number={2},
  pages={146--160},
  year={1972},
  publisher={SIAM}
}

@article{scherrer2013improved,
  title={Improved and generalized upper bounds on the complexity of policy iteration},
  author={Scherrer, Bruno},
  journal={Advances in Neural Information Processing Systems},
  volume={26},
  year={2013}
}

@article{ye2005new,
  title={A new complexity result on solving the Markov decision problem},
  author={Ye, Yinyu},
  journal={Mathematics of Operations Research},
  volume={30},
  number={3},
  pages={733--749},
  year={2005},
  publisher={INFORMS}
}

\end{document}